\documentclass[a4paper]{amsart}
\usepackage[utf8]{inputenc}
\usepackage{amsmath,amssymb,amsthm}
\usepackage[english]{babel}
\usepackage{eucal}
\usepackage{graphicx}
\usepackage{bm}
\usepackage{esint}
\usepackage{color}
\usepackage{url}

\newcommand{\rot}{\operatorname{rot}}
\newcommand{\ddiv}{\operatorname{div}}
\newcommand{\diam}{\operatorname{diam}}
\newcommand{\trace}{\operatorname{tr}}
\newcommand{\ns}{^{\textup{\tiny NS}}}
\newcommand{\ls}{^{\textup{\tiny LS}}}

\newtheorem{proposition}{Proposition}
\newtheorem{lemma}[proposition]{Lemma}
\newtheorem{assumption}{Assumption}
\theoremstyle{remark}
\newtheorem{remark}[proposition]{Remark}

\begin{document}
\title[Nondivergence PDEs]{Variational formulation and numerical analysis of linear elliptic
      equations in nondivergence form with Cordes coefficients}

\author[D.~Gallistl]{Dietmar Gallistl}
\address[D.~Gallistl]{Institut für Angewandte und Numerische Mathematik,
           Karlsruher Institut f\"ur Technologie,
           Englerstr.~2,
           D-76131 Karlsruhe, Germany.
         }
\email{gallistl(at)kit(dot)edu}
\date\today

\begin{abstract}
This paper studies formulations of
second-order elliptic partial differential
equations in nondivergence form on convex domains
as equivalent variational problems.
The first formulation is that of
Smears \& S\"uli [SIAM J.\ Numer.\ Anal.\ 51(2013), pp.\ 2088--2106.],
and
the second one is a new symmetric formulation based on a least-squares
functional.
These formulations enable the use of
standard finite element techniques for variational problems
in subspaces of $H^2$ as well as mixed finite element methods
from the context of fluid computations.
Besides the immediate quasi-optimal a~priori error bounds,
the variational setting allows for a~posteriori
error control with explicit constants
and adaptive mesh-refinement.
The convergence of an adaptive algorithm is proved.
Numerical results on uniform and adaptive meshes are included.
\end{abstract}



\maketitle

\section{Introduction}

Let $\Omega\subseteq\mathbb R^d$ be an open, bounded, convex
polytope for $d\in\{2,3\}$.
This article deals with the numerical approximation of strong
solutions $u\in H^1_0(\Omega)\cap H^2(\Omega)$ to the second-order
elliptic partial differential equation (PDE) 
\begin{equation}\label{e:pdeL}
L(u) = f \quad\text{in }\Omega\qquad u=0\quad\text{on }\partial\Omega
\end{equation}
where $f\in L^2(\Omega)$ is a given square-integrable function and
the operator $L$ has nondivergence form.
More precisely, it is given through
\begin{equation}\label{e:defL}
L(v):= A:D^2 v
:=
\sum_{j,k=1}^d A_{jk}\partial^2_{jk}v
 \qquad\text{for any }v\in V:=H^1_0(\Omega)\cap H^2(\Omega).
\end{equation}
In the case that the coefficient $A$ satisfies certain smoothness
assumptions, it is known that \eqref{e:pdeL} can be converted into a
second-order equation in divergence form through the product rule.
If $A$ is merely an essentially bounded tensor, such a reformulation
is not valid and variational formulations of \eqref{e:pdeL} are 
less obvious. 
It is proved in \cite{SmearsSueli2013} that
the unique solvability is assured through the Cordes
condition \cite{Cordes1956,MaugeriPalagachevSoftova2000}
described in \S\ref{s:formulation} below.
The first fully analyzed numerical scheme suited for $L^\infty$ 
Cordes coefficients
was suggested and analyzed in \cite{SmearsSueli2013} and belongs to
the class of discontinuous Galerkin methods.
It was successfully applied in
\cite{SmearsSueli2014,SmearsSueli2016} to fully-nonlinear
Hamilton--Jacobi--Bellman equations.
Further works on discontinuous Galerkin methods for nondivergence 
form problems
\cite{FengHenningsNeilan2016,FengNeilanSchnake2016}
focus on error
estimates in $W^{k,p}$ norms for the case of continuous coefficients.
Other approaches include the discrete Hessian method of
\cite{LakkisPryer2011} and the two-scale method of
\cite{NochettoZhang2014}. The latter work is based on the 
integro-differential approach of \cite{CaffarelliSilvestre2010}
and focusses on $L^\infty$ error estimates.

This paper 
studies variational formulations of \eqref{e:defL}
for the case of discontinuous coefficients
satisfying the Cordes condition.
The formulation seeks $u\in V$ such that
\begin{equation*}
 (A:D^2 u, \tau(\nabla v))_{L^2(\Omega)} = (f,\tau(\nabla v))_{L^2(\Omega)}
\quad\text{for all } v\in V,
\end{equation*}
where the operator 
$\tau:H^1(\Omega;\mathbb R^d)\to L^2(\Omega)$ 
acts on the test functions.
In this work, two possible options are discussed.
The choice $\tau=\tau\ns:=\gamma\ddiv\bullet$ (for the function $\gamma$
defined in \eqref{e:gammadef} below) leads to the
nonsymmetric formulation
of \cite{SmearsSueli2013}.
The second possibility is 
$\tau=\tau\ls:=A:D\bullet$ which results in a
symmetric problem that
turns out to be the Euler--Lagrange equation
for the minimization of the functional
$\| A:D^2 v - f\|_{L^2(\Omega)}^2$.
The superscripts $\mathrm{NS}$ and $\mathrm{LS}$ stand for
`non-symmetric' and `least-squares', respectively, owing to the 
properties of the individual method.
The variational formulations naturally allow the use
of $C^1$-conforming finite element
methods \cite{Ciarlet1978}.
 Since $C^1$ finite elements are sometimes considered
impractical, alternative discretization techniques are desirable.
We apply the recently proposed mixed formulation \cite{Gallistl2016} 
to the present problem. Its formulation involves function
spaces similar to those employed for the Stokes equations.
In the sense of the least-squares functional, the minimization
problem is restated as the minimization of
$\| A:D \phi - f\|_{L^2(\Omega)}^2$ over all vector-valued $H^1$ functions
with vanishing tangential trace subject to the constraint that
$\rot \phi = 0$.
While the continuous formulations are equivalent, the latter 
can be discretized with $H^1$-conforming finite elements
in the framework of saddle-point problems
\cite{BoffiBrezziFortin2013,Gallistl2016}.
In the discrete formulation, the structure of the differential operator
requires the incorporation of an additional stabilization term.
This is mainly due to the fact that $L(v)$ in the $L^2$ norm is bounded
from below by the norm of the Laplacian $\Delta v$ rather than the 
full Hessian tensor $D^2 v$. This is also the reason why the application
of nonconforming schemes is not as immediate as for the usual
biharmonic equation. Indeed, nonconforming finite element spaces
may contain piecewise harmonic functions and, thus, it is not
generally possible
to bound the piecewise Laplacian from below by the piecewise Hessian
unless further stabilization terms are included.
For example, the divergence theorem readily implies that
three out of the six local basis function of the Morley finite element
\cite{Ciarlet1978} are harmonic.
The conforming and mixed finite element formulations 
presented here lead to 
quasi-optimal a~priori error estimates
and give rise to natural
a~posteriori error estimates  based on strong $L^2$ volume residuals
where on any element of the finite element partition the
residual reads
$\|A:D^2 u_h - f\|_{L^2(T)}^2$
for the conforming finite element solution $u_h$ (with an 
analogous formula for the mixed discretization).
Since this residual equals $\|A:D^2 (u_h - u)\|_{L^2(T)}^2$,
it immediately leads to reliable and efficient estimates with
explicit constants (depending solely on the data).
This error estimator can be employed for guiding a
self-adaptive refinement procedure.
This work focusses on $h$-adaptivity and does not address
a local adaptation of the polynomial degree as 
in \cite{SmearsSueli2013}.
For the suggested class of discretizations, the convergence
of the adaptive algorithm can be proved.
Since the proof utilizes a somehow indirect argument
(similar to that of \cite{MorinSiebertVeeser2008}), no convergence
rate is obtained.
The performance of the adaptive mesh-refinement procedure
is numerically studied in the experiments of this paper.

\medskip
The remaining parts of this article are as follows: 
\S\ref{s:formulation} revisits the unique solvability results
of \cite{SmearsSueli2013} and presents the
variational formulations; \S\ref{s:FEM} presents the a~priori and
a~posteriori error estimates for finite element discretizations.
The convergence analysis of an adaptive algorithm follows
in \S\ref{s:afem}. 
Numerical experiments are presented in \S\ref{s:num}.
The remarks of \S\ref{s:conclusion} conclude the paper.

\medskip
Standard notation on function spaces applies throughout
this article.
Lebesgue and Sobolev functions with values in $\mathbb R^d$
are denoted by $L^2(\Omega;\mathbb R^d)$,
$H^1(\Omega;\mathbb R^d)$, etc.
The $d\times d$ identity matrix is denoted by $I_{d\times d}$.
The inner product of real-valued $d\times d$ matrices $A$, $B$
is denoted by $A:B = \sum_{j,k=1}^d A_{jk} B_{jk}$.
The Frobenius norm of a $d\times d$ matrix $A$ is denoted by
$\lvert A\rvert:=\sqrt{A:A}$;
the trace reads $\trace A$.
For vectors, $\lvert\cdot\rvert$ refers to the Euclidean
length.
The rotation 
(often referred to as $\operatorname{curl} v$ or $\nabla\wedge v$)
of a vector field $v$ is denoted by $\rot v$.
The union of a collection $\mathcal X$ of subsets of $\mathbb R^d$
is indicated by the symbol $\cup$ without index and reads
$
 \cup\mathcal X 
 := \{ x\in\mathbb R^d : x\in X\text{ for some } X\in\mathcal X \}.
$

\section{Problem setting and variational formulations}\label{s:formulation}
This section lists some conditions for the unique solvability 
of \eqref{e:pdeL} and proceeds with the variational formulations.
Throughout this article it is assumed that the coefficient
$A\in L^\infty(\Omega;\mathbb R^{d\times d})$
is uniformly elliptic, that is, there exist constants
$0<\alpha_1\leq\alpha_2<\infty$ such that
\begin{equation}\label{e:ellipticity}
\alpha_1
\leq
\inf_{\xi\in\mathbb R^d, |\xi|=1}
\xi^* A \xi
\leq
\sup_{\xi\in\mathbb R^d, |\xi|=1}
\xi^* A \xi
\leq
\alpha_2
\quad\text{almost everywhere in }\Omega.
\end{equation}
Assume furthermore that
there exists some $\varepsilon\in (0,1]$ such that
\begin{equation}\label{e:cordes}
\lvert A \rvert^2  \big/ (\trace A)^2
\leq
1\big/(d-1+\varepsilon)
\quad\text{almost everywhere in }\Omega.
\end{equation}
Assumption \eqref{e:cordes} is called the \emph{Cordes condition}
\cite{Cordes1956,MaugeriPalagachevSoftova2000,SmearsSueli2013}.
Define the function $\gamma$ by
\begin{equation}\label{e:gammadef}
\gamma := \trace(A) \big/ \lvert A\rvert^2 .
\end{equation}
While in the planar case, $d=2$, the Cordes condition is
implied by \eqref{e:ellipticity}, it is an essential condition for
$d\geq 3$ and its absence may lead to ill-posedness of the 
PDE \eqref{e:pdeL}
\cite{MaugeriPalagachevSoftova2000,Safonov1999}.
The uniform ellipticity \eqref{e:ellipticity} implies that
$\gamma$ is uniformly bounded from below by some positive
constant $\gamma_0$ \cite{SmearsSueli2013}.
The following result
can be found in 
\cite{MaugeriPalagachevSoftova2000,SmearsSueli2013}.

\begin{lemma}\label{l:sueli}
Let $A\in L^\infty(\Omega;\mathbb R^{d\times d})$
satisfy \eqref{e:ellipticity} and \eqref{e:cordes}.
Then, almost everywhere in $\Omega$, 
the following estimate
holds for any $B\in \mathbb R^{d\times d}$,
\begin{equation*}
|(\gamma A - I_{d\times d} ): B | 
=
|\gamma A : B - \trace B| 
\leq \sqrt{1-\varepsilon} \, |B| 
\end{equation*}
as well as
$
|\gamma A- I_{d\times d}|
\leq
\sqrt{1-\varepsilon}.
$
\end{lemma}
\begin{proof}
See, e.g., the proof of \cite[Lemma~1]{SmearsSueli2013}.
\end{proof}

The triangle inequality shows that $A$
almost everywhere satisfies 
for any $B\in \mathbb R^{d\times d}$
\begin{equation}\label{e:ABpointwise}
\gamma |A:B|
  \geq 
  |\trace  B | -
  |(\gamma A - I_{d\times d} ): B | .
\end{equation}
The space of $H^1$ vector fields with vanishing tangential
trace reads
\begin{equation}\label{e:Wdef}
W:=\{v\in H^1(\Omega;\mathbb R^d) :
        \text{the tangential trace of } v \text{ on }
        \partial\Omega \text{ vanishes}
    \} .
\end{equation}
For the analysis of the formulations below, it
it is useful to note that, on convex domains,
the following estimate holds 
\cite[Thm.~2.3]{CostabelDauge1999}
\begin{equation}\label{e:divrotestimate}
\| D w \|_{L^2(\Omega)}^2
\leq
\| \rot w \|_{L^2(\Omega)}^2
+
\| \ddiv w \|_{L^2(\Omega)}^2 
\quad\text{for any }w\in W
\end{equation}
(on polytopes even with equality).
For any $w\in W$ with $\rot w=0$,
the combination of \eqref{e:ABpointwise}
for $B=Dw$ with Lemma~\ref{l:sueli} and \eqref{e:divrotestimate}
results in
\begin{equation}\label{e:ABlowerbound}
\|\gamma  A:Dw\|_{L^2(\Omega)} 
  \geq 
  (1-\sqrt{1-\varepsilon} ) \, \| \ddiv w \|_{L^2(\Omega)}
  \geq 
  (1-\sqrt{1-\varepsilon} ) \, \| D w \|_{L^2(\Omega)}.
\end{equation}
Similar calculations
(already carried out in \cite{SmearsSueli2013})
with Lemma~\ref{l:sueli} and \eqref{e:divrotestimate} prove 
for any $w\in W$ with $\rot w=0$ that
\begin{equation}\label{e:lowerboundSmearsSueli}
\begin{aligned}
 (A:D w, \gamma\ddiv w)_{L^2(\Omega)}
 &
 =
 \|\ddiv w\|_{L^2(\Omega)}^2
 +((\gamma A-I_{d\times d}) : Dw, \ddiv w)_{L^2(\Omega)}
 \\
 &
 \geq 
 (1-\sqrt{1-\varepsilon}) \|\ddiv w\|_{L^2(\Omega)}^2
\geq (1-\sqrt{1-\varepsilon}) \|D w\|_{L^2(\Omega)}^2 .
\end{aligned}
\end{equation}

We proceed with the description of the variational setting.
Define the space $V:=H^1_0(\Omega)\cap H^2(\Omega)$.
An application of \eqref{e:divrotestimate} shows that
the $L^2$ norm of the Hessian of
any $v\in V$ is controlled by the norm of the Laplacian
\cite{CostabelDauge1999,Grisvard1985,SmearsSueli2013}
\begin{equation}\label{e:LaplHess}
\| D^2 v \|_{L^2(\Omega)} \leq \| \Delta v \|_{L^2(\Omega)}
\quad\text{for any }v\in V.
\end{equation}
Define the operators $\tau\ns,\tau\ls:H^1(\Omega;\mathbb R^d)\to L^2(\Omega)$
by
\begin{equation*}
 \begin{aligned}
  \tau\ns(\phi):=\gamma\ddiv\phi  
  \quad\text{and}\quad
  \tau\ls(\phi):=A:D\phi
  \qquad\text{for any }\phi\in H^1(\Omega;\mathbb R^d).
 \end{aligned}
\end{equation*}
As mentioned in the introduction, each of these operators corresponds
to a specific choice of test functions in a variational formulation
and, thus, constitutes a class of numerical methods.
The variational problem seeks $u\in V$ such that
\begin{equation}\label{e:varform}
(A:D^2 u, \tau(\nabla v))_{L^2(\Omega)} = (f,\tau(\nabla v))_{L^2(\Omega)}
\quad\text{for all }v\in V
\end{equation}
for $\tau=\tau\ns$ (the nonsymmetric formulation of
\cite{SmearsSueli2013}) or $\tau=\tau\ls$ (the least-squares formulation
proposed here).
The lower bound \eqref{e:lowerboundSmearsSueli}
with $w=\nabla v$ implies that \eqref{e:varform} is coercive for
$\tau=\tau\ns$.
The lower bound \eqref{e:ABlowerbound}
with $w=\nabla v$ implies for
any $v\in V$ that
\begin{equation*}
\|\gamma\|_{L^\infty(\Omega)}^2
(A:D^2 v, \tau\ls(\nabla v))_{L^2(\Omega)}
\geq
\|\gamma A:D^2 v\|_{L^2(\Omega)}^2
\geq 
(1-\sqrt{1-\varepsilon} )^2
\| D^2 v\|_{L^2(\Omega)}^2.
\end{equation*}
Thus, $(A:D^2\bullet,A:D^2\bullet)_{L^2(\Omega)}$
 is an inner product on $V$ with
\begin{align}
\label{e:H2prodlowerbound}
\|A:D^2v\|_{L^2(\Omega)} 
&
\geq
c(\gamma,\varepsilon)
\| D^2 v\|_{L^2(\Omega)}
\quad\text{for any } v\in V
\\
\label{e:cdef}
\text{where}\qquad\qquad\qquad
 c(\gamma,\varepsilon) 
&:=
 (1-\sqrt{1-\varepsilon}) / \|\gamma\|_{L^\infty(\Omega)}.
\end{align}
This yields well-posedness of \eqref{e:varform}
for $\tau=\tau\ls$.
The following result proves the equivalence of
\eqref{e:pdeL} and \eqref{e:varform}.

\begin{proposition}
Let $\tau=\tau\ns$ or $\tau=\tau\ls$.
A function $u\in V$ solves \eqref{e:pdeL} strongly in $L^2(\Omega)$
if and only if it solves the variational form \eqref{e:varform}.
\end{proposition}
\begin{proof}
For the choice $\tau=\tau\ns$, the assertion was proved in
\cite[proof of Thm.~3]{SmearsSueli2013}, and it remains to consider the case
$\tau=\tau\ls$.
It is immediate that \eqref{e:pdeL} implies \eqref{e:varform}.
For the converse direction it is enough to note that
\eqref{e:varform} is the Euler--Lagrange equation of the 
convex minimization problem
\begin{equation}\label{e:leastsquaresADuf}
 u\in \operatorname*{arg\,min}_{v\in V}
   \| A:D^2 v - f\|_{L^2(\Omega)}^2 .
\end{equation}
Since \eqref{e:ellipticity} and \eqref{e:cordes} imply
that \eqref{e:pdeL} is uniquely solvable on convex domains
\cite{SmearsSueli2013},
this establishes the equivalence.
\end{proof}

Formulation \eqref{e:varform} is variational
and, thus, suited for approximation with the finite element
method (FEM). Standard finite elements will be discussed in
Subsection~\ref{ss:standardfem}.
Since the construction of $H^2$-conforming finite elements is
rather cumbersome, mixed formulations appear as an attractive
alternative.
To state the mixed formulation recently proposed in
\cite{Gallistl2016}, recall the definition of $W$
from \eqref{e:Wdef} and define the space
\begin{equation*}
Q:= \begin{cases}
       \{q\in L^2(\Omega) : 
        \int_\Omega q\,dx =0 \}& \text{if } d=2, \\
       \{q\in L^2(\Omega;\mathbb R^3) : 
        \ddiv q = 0 \text { in } \Omega
        \text{ and } q\cdot\nu|_{\partial\Omega}=0
        \text{ on }\partial\Omega\}& \text{if } d=3 .
       \end{cases}
\end{equation*}
Here $\nu$ denotes the outer unit normal of the domain $\Omega$.
Define the bilinear forms
$a_\tau:W\times W\to \mathbb R$ (for $\tau=\tau\ns$ or
$\tau=\tau\ls$) and
 $b:W\times Q \to \mathbb R$
by
\begin{equation*}
\begin{aligned}
a_\tau(v,z) & := (A:Dv, \tau(z))_{L^2(\Omega)} 
\quad
&&\text{for any } (v,z)\in W\times W,\\
b(v,q) &:= (\rot v, q)_{L^2(\Omega)}
&&\text{for any } (v,q)\in W\times Q.
\end{aligned}
\end{equation*}
The mixed formulation of \eqref{e:varform} is to seek
$(u,w,p)\in H^1_0(\Omega) \times W \times Q$ such that
\begin{subequations}\label{e:saddlepoint}
\begin{align}\label{e:saddlepoint_a}
(\nabla u - w,\nabla z)_{L^2(\Omega)} &=0
\quad&&\text{for all } z\in H^1_0(\Omega),\\
\label{e:saddlepoint_b}
a_\tau(w,v) + b(v,p)  &= (f,\tau(v))_{L^2(\Omega)}
&&\text{for all } v\in W,
\\
\label{e:saddlepoint_c}
b(w,q)    &= 0
&&\text{for all }q\in Q.
\end{align}
\end{subequations}

For the analysis of the well-posedness of
\eqref{e:saddlepoint_b}--\eqref{e:saddlepoint_c},
recall estimate
\eqref{e:lowerboundSmearsSueli} for $\tau=\tau\ns$
and 
\eqref{e:ABlowerbound} for $\tau=\tau\ls$,
which imply that the form 
$a_\tau$ is coercive on the subspace
of $W$ consisting of rotation-free vector fields,
namely, for all $v\in W$ with $\rot v=0$,
\begin{equation}\label{e:lowerboundAmixed}
\begin{aligned}
(1-\sqrt{1-\varepsilon}) \| Dv\|_{L^2(\Omega)}^2
&\leq 
a_{\tau\ns}(v,v)\leq \|A\|_{L^\infty(\Omega)} \|\gamma\|_{L^\infty(\Omega)}
 \| Dv\|_{L^2(\Omega)}^2,
\\
c(\gamma,\varepsilon)^2 \| Dv\|_{L^2(\Omega)}^2
&\leq 
a_{\tau\ls}(v,v)\leq \|A\|_{L^\infty(\Omega)}^2 \| Dv\|_{L^2(\Omega)}^2.
\end{aligned}
\end{equation}
Since there exists a constant $\beta>0$ such that the 
 following inf-sup condition is valid
\begin{equation}\label{e:infsup}
 \beta \leq  
  \inf_{q\in Q\setminus\{0\}} \sup_{v\in W\setminus\{0\}} 
  b(v,q) \big/ (\|Dv \|_{L^2(\Omega)}\|q \|_{L^2(\Omega)}),
\end{equation}
problem \eqref{e:saddlepoint_b}--\eqref{e:saddlepoint_c}
(and thus \eqref{e:saddlepoint}) is uniquely solvable
\cite{BoffiBrezziFortin2013}. 
The stability \eqref{e:infsup} (employed in \cite{Gallistl2016})
is based on a regularized decomposition given in \cite{LouMcIntosh2005}
which is stronger than the classical Helmholtz decomposition
\cite{GiraultRaviart1986}.

\begin{proposition}
Let $\tau=\tau\ns$ or $\tau=\tau\ls$.
Problems \eqref{e:varform} and \eqref{e:saddlepoint}
are equivalent in the following sense.
If $u\in V$ solves \eqref{e:varform}, then there exists
$p\in Q$ such that $(u,\nabla u,p)$ solves \eqref{e:saddlepoint}.
If, conversely, $(u,w,p)\in H^1_0(\Omega)\times W\times Q$
solves \eqref{e:saddlepoint}, then $u$ belongs to $V$ and
solves \eqref{e:varform} with $w = \nabla u$.
\end{proposition}
\begin{proof}
The proof is essentially contained in \cite{Gallistl2016}.
For completeness, it is sketched here.
It is easily verified that $w:=\nabla u$
is rotation-free, i.e., $\rot w= 0$. By \eqref{e:infsup} there exists
a Lagrange multiplier $p\in Q$ such that system
\eqref{e:saddlepoint} is satisfied.
Let, conversely, $w$ satisfy
\eqref{e:saddlepoint_b}--\eqref{e:saddlepoint_c}.
Since $\rot w=0$ and $\Omega$ is convex,
there exists a potential $\phi\in H^1_0(\Omega)$
with $\nabla\phi = w$.
By \eqref{e:saddlepoint_a}, the difference $u-\phi$ satisfies
the homogeneous Laplace equation with zero Dirichlet conditions,
hence $u=\phi$ and $w=\nabla u$.
\end{proof}

\begin{remark}
It is not difficult to see that the Lagrange multiplier $p$ equals
zero in the continuous setting. This property will not be preserved
by typical discretizations.
\end{remark}

\begin{remark}
In the case that the operator $L$ has the form \eqref{e:defL},
the system \eqref{e:saddlepoint}
decouples into a Stokes-type problem plus the post-processing
for the primal variable.
For more general equations involving zeroth-order terms,
see the comments in Section~\ref{s:conclusion}.
\end{remark}

\section{Finite element discretization}\label{s:FEM}

This section presents conforming and mixed finite element 
discretizations and their error analysis
for the problems of \S\ref{s:formulation}.

\subsection{Conforming discretization}\label{ss:standardfem}

The variational formulation \eqref{e:varform} immediately allows
stable discrete formulations with conforming finite elements.
Let $V_h\subseteq V$ be a closed subspace.
The discrete problem is to seek $u_h\in V_h$ such that
\begin{equation}\label{e:discrprobconf}
(A:D^2 u_h, \tau(\nabla v_h))_{L^2(\Omega)} 
= (f,\tau(\nabla v_h))_{L^2(\Omega)}
\quad\text{for all } v_h\in V_h,
\end{equation}
for $\tau=\tau\ns$ or $\tau=\tau\ls$.
Although the discrete solution $u_h$ depends on the choice of $\tau$,
this dependence will not be indicated by an additional
index on $u_h$. The same applies to the mixed scheme below.
The following result states the error analysis.
\begin{proposition}\label{p:confanalysis}
Problem \eqref{e:discrprobconf} is uniquely solvable.
The error $u-u_h$ for the solution $u\in V$ to \eqref{e:varform}
and the discrete solution $u_h\in V_h$ to \eqref{e:discrprobconf}
satisfies
\begin{equation*}
\| D^2 (u-u_h)\|_{L^2(\Omega)}
\leq 
c(\gamma,\varepsilon)^{-1}\|A\|_{L^\infty(\Omega)}
\inf_{v_h\in V_h} \| D^2 (u-v_h)\|_{L^2(\Omega)}.
\end{equation*}
Globally, the following reliable a~posteriori error estimate holds
\begin{equation*}
c(\gamma,\varepsilon) \|D^2(u-u_h)\|_{L^2(\Omega)}
\leq
\| A:D^2 u_h - f\|_{L^2(\Omega)}.
\end{equation*}
Furthermore, for any open subdomain $\omega\subseteq\Omega$,
the following local efficiency is valid
\begin{equation*}
\| A:D^2 u_h - f\|_{L^2(\omega)}
\leq 
\|A\|_{L^\infty(\omega)} \| D^2 (u-u_h)\|_{L^2(\omega)}.
\end{equation*}

\end{proposition}
\begin{proof}
For $\tau=\tau\ns$, the
a~priori result follows from combining C\'ea's Lemma
\cite{Braess2007} with the lower bound \eqref{e:lowerboundSmearsSueli}.
For $\tau=\tau\ls$, it follows with \eqref{e:H2prodlowerbound}
and similar arguments. Therein, the symmetry of the formulation 
allows the use of equivalence of norms which
implies the stated constant while a general C\'ea-type estimate would
result in the square of that constant.
The reliability result follows from the 
lower bound 
\eqref{e:lowerboundSmearsSueli} and \eqref{e:ABlowerbound}
for $\tau=\tau\ns$ and $\tau=\tau\ls$, respectively,
and the fact that $f=D^2u$ in the $L^2$ sense.
The latter fact also proves the efficiency estimate.
\end{proof}

Several instances of finite-dimensional piecewise polynomial
conforming subspaces
$V_h$ are known \cite{Ciarlet1978}. 
In the numerical experiments of Section~\ref{s:num}, the 
performance of the Bogner--Fox--Schmit (BFS) finite element under 
adaptive mesh-refinement based on the a~posteriori error estimator
of Proposition~\eqref{p:confanalysis}
is empirically studied.

\subsection{Mixed discretization}\label{ss:mixedFEM}
The conformity assumption
$V_h\subseteq H^2(\Omega)$ requires $C^1$ continuity
and results in rather complicated local constructions.
An alternative discretization is based on the formulation
\eqref{e:saddlepoint} and mixed Stokes-type finite elements
\cite{Gallistl2016}.
Suppose that $W_h\subseteq W$ and $Q_h\subseteq Q$ are closed
subspaces that satisfy for some positive constant
$\tilde\beta$ that
\begin{equation}\label{e:infsup_d}
 \tilde\beta\leq
  \inf_{q_h\in Q_h\setminus\{0\}} \sup_{v_h\in W_h\setminus\{0\}} 
 b(v_h,q_h)\big/(\|Dv_h \|_{L^2(\Omega)}\|q_h \|_{L^2(\Omega)}).
\end{equation}
Since, in general, the property $b(v_h,q_h)=0$ for all $q_h\in Q_h$
does not imply that $\rot v_h =0$,
the argument \eqref{e:divrotestimate} is not applicable and
coercivity of $a$ on the kernel of $b$ requires
stabilization.
The proposed stabilization is as follows.
Define the constant

\begin{equation*}
c_\lambda^\tau(\gamma,\varepsilon)
:=
\begin{cases}
\sqrt{1-\frac{\lambda^2+1-\varepsilon}{2\lambda} }
&\text{if }\tau=\tau\ns,
\\
\frac{1-\sqrt{1-\varepsilon}}
     {\|\gamma\|_{L^\infty(\Omega)}\sqrt{1+\lambda}}
= c(\gamma,\varepsilon)/\sqrt{1+\lambda}
&\text{if }\tau=\tau\ls
\end{cases}
\end{equation*}
for a given parameter $\lambda>0$ which in the case 
$\tau=\tau\ns$ is subject to the additional constraint
$|\lambda-1|<\sqrt{\varepsilon}$.
It should also be noted that in the case $\tau=\tau\ns$
the constant $c_\lambda^\tau(\gamma,\varepsilon)$ is
independent of $\gamma$. The notation, however,
is maintained for the sake of a unified presentation.
Define furthermore the stabilization parameter
$$
\sigma_\lambda^\tau(\gamma,\varepsilon)
:=
\begin{cases}
\sqrt{c_\lambda^\tau(\gamma,\varepsilon)^2 
+(1-\varepsilon)/(2\lambda)}
= \sqrt{1-\lambda/2}
&\text{if } \tau=\tau\ns,
\\
\sqrt{%
c_\lambda^\tau(\gamma,\varepsilon)^2 
+ \frac{(1+1/\lambda)(1-\varepsilon)}
       {(1+\lambda) \|\gamma\|_{L^\infty(\Omega)}^2}
}
&\text{if } \tau=\tau\ls.
\end{cases}
$$
Define the enriched bilinear form
$\tilde a_\tau$ for any $v,z\in W$ through
\begin{equation*}
\tilde a_\tau(v,z) 
:= 
a_\tau(v,z) 
 + 
\sigma_\lambda^\tau(\gamma,\varepsilon)^2
(\rot v,\rot z )_{L^2(\Omega)}.
\end{equation*}

Let $S_h\subseteq H^1_0(\Omega)$ be a closed subspace.
The discrete mixed system seeks 
$(u_h,w_h,p_h)\in S_h \times W_h \times Q_h$ such that
\begin{subequations}\label{e:discrsaddlepoint}
\begin{align}\label{e:discrsaddlepoint_a}
(\nabla u_h - w_h,\nabla z_h)_{L^2(\Omega)} &=0
\quad&&\text{for all } z_h\in S_h,\\
\label{e:discrsaddlepoint_b}
 \tilde a_\tau(w_h,v_h)
+ b(v_h,p_h)  &= 
(f,\tau(v_h))_{L^2(\Omega)}
&&\text{for all } v_h\in W_h,
\\
\label{e:discrsaddlepoint_c}
b(w_h,q_h)    &= 0
&&\text{for all }q_h\in Q_h.
\end{align}
\end{subequations}

The following proposition states well-posedness and 
error estimates for \eqref{e:discrsaddlepoint} 
in the case $\tau=\tau\ns$.

\begin{proposition}\label{p:mixedanalysisNS}
Let $\tau=\tau\ns$.
For any $\lambda>0$ such that 
$|\lambda-1|\leq\sqrt{\varepsilon}$,
problem~\eqref{e:discrsaddlepoint} admits a unique solution
$(u_h,w_h,p_h)\in S_h \times W_h \times Q_h$.
It satisfies the error estimate
\begin{equation*}
\|D^2u-D w_h\|_{L^2(\Omega)}
\leq
C(\lambda,\gamma,\varepsilon,\tau)
\inf_{v_h\in W_h}
\|D^2u-D v_h\|_{L^2(\Omega)}
\end{equation*}
where
$
C(\lambda,\gamma,\varepsilon,\tau)
=
4c_\lambda^\tau(\gamma,\varepsilon)^{-2}\tilde\beta^{-1}
(\|\gamma\|_{L^\infty(\Omega)} \|A\|_{L^\infty(\Omega)}
  + \sigma_\lambda^\tau(\gamma,\varepsilon)^2) .
$
Moreover the following reliable a~posteriori error estimate holds
for any $\mu>0$
with 
$\mu\|\gamma\|_{L^\infty(\Omega)}^2
 \leq 2c_\lambda^\tau(\gamma,\varepsilon)^2$,
\begin{equation*}
\begin{aligned}
&
\sqrt{c_\lambda^\tau(\gamma,\varepsilon)^2 - 2^{-1}\mu\|\gamma\|_{L^\infty(\Omega)}^2}
 \|D^2 u -Dw_h \|_{L^2(\Omega)}
\\
&\qquad\qquad
\leq 
\sqrt{(2\mu)^{-1}\| A:D w_h -f \|_{L^2(\Omega)}^2
+
\sigma_\lambda^\tau(\gamma,\varepsilon)^2
 \|\rot w_h \|_{L^2(\Omega)}^2}
.
\end{aligned}
\end{equation*}
For any open subdomain $\omega\subseteq \Omega$ we have the efficiency
\begin{equation*}
\begin{aligned}
&
\sqrt{(2\mu)^{-1}
\| A:D w_h -f \|_{L^2(\omega)}^2
+\sigma_\lambda^\tau(\gamma,\varepsilon)^2 \|\rot w_h \|_{L^2(\omega)}^2}
\\
&\qquad\qquad\qquad\qquad
\leq 
\sqrt{(2\mu)^{-1}\|A\|_{L^\infty(\omega)}^2 
+ \sigma_\lambda^\tau(\gamma,\varepsilon)^2}
\|D^2 u -Dw_h \|_{L^2(\omega)}.
\end{aligned}
\end{equation*}
\end{proposition}
\begin{proof}
Let $v\in W$.
The argument from the first line of \eqref{e:lowerboundSmearsSueli}
together with Lemma~\ref{l:sueli} and \eqref{e:divrotestimate}
leads to
\begin{equation*}
\begin{aligned}
 &
 (A:D v, \gamma\ddiv v)_{L^2(\Omega)}
\\
 &\quad
 =
 \|\ddiv v\|_{L^2(\Omega)}^2
 +((\gamma A-I_{d\times d}) : Dv, \ddiv v)_{L^2(\Omega)}
 \\
 &\quad
 \geq 
 \|\ddiv v\|_{L^2(\Omega)}^2
 -\sqrt{1-\varepsilon}\|\ddiv v\|_{L^2(\Omega)}
    \sqrt{\|\ddiv v\|_{L^2(\Omega)}^2+\|\rot v\|_{L^2(\Omega)}^2}.
\end{aligned}
\end{equation*}
For any $\lambda>0$,
the Young inequality bounds the right-hand side from below by
\begin{equation*}
 \big(1-\lambda/2-(1-\varepsilon)/(2\lambda)\big)
  \|\ddiv v\|_{L^2(\Omega)}^2
 -(1-\varepsilon)/(2\lambda)\|\rot v\|_{L^2(\Omega)}^2.
\end{equation*}
Elementary calculations with the foregoing two displayed expressions
therefore lead,
after adding 
$c_\lambda^\tau(\gamma,\varepsilon)^2\|\rot v\|_{L^2(\Omega)}^2$,
 to the coercivity
\begin{equation}\label{e:coercivityNSmixed}
 c_\lambda^\tau(\gamma,\varepsilon)^2 \|Dv\|_{L^2(\Omega)}^2
\leq
\tilde{a}_\tau(v,v)
\quad\text{for any } v\in W.
\end{equation}
The constant  $c_\lambda^\tau(\gamma,\varepsilon)^2$
is positive if and only if $|\lambda-1|<\sqrt{\varepsilon}$.
This and \eqref{e:infsup_d} yield well-posedness. The a~priori
error estimate follows from the mixed finite element theory
\cite{BoffiBrezziFortin2013}, see
\cite[Thm.~5.2.2]{BoffiBrezziFortin2013} for the precise constant.
In particular, the best-approximation error of $p$ does not appear
in the error bound because that term equals zero due to $p=0$.
The proof of the reliability estimate employs the coercivity
\eqref{e:coercivityNSmixed},
the $L^2$ identity $A:D^2u=f$,
as well as the Cauchy and Young inequalities
with an arbitrary parameter $\mu>0$,
\begin{equation*}
\begin{aligned}
&
   c_\lambda^\tau(\gamma,\varepsilon)^2 \|D(\nabla u- w_h)\|_{L^2(\Omega)}^2
\leq
\tilde{a}_\tau(\nabla u- w_h,\nabla u- w_h)
\\
&
\qquad
\leq 
(2\mu)^{-1}
 \|A:Dw_h - f\|_{L^2(\Omega)}^2
+ 2^{-1}\mu \|\gamma\|_{L^\infty(\Omega)}^2 \|D(\nabla u- w_h)\|_{L^2(\Omega)}^2
\\
&
\qquad\qquad
+
 \sigma_\lambda^\tau(\gamma,\varepsilon)^2 \|\rot w_h\|_{L^2(\Omega)}^2.
\end{aligned}
\end{equation*}
This implies the stated reliability.
The efficiency follows from $A:D^2u=f$ in $L^2$.
\end{proof}

The following proposition states well-posedness and 
error estimates for \eqref{e:discrsaddlepoint} 
in the case $\tau=\tau\ls$.

\begin{proposition}\label{p:mixedanalysisLS}
Let $\tau=\tau\ls$.
For any $\lambda>0$,
problem~\eqref{e:discrsaddlepoint} admits a unique solution
$(u_h,w_h,p_h)\in S_h \times W_h \times Q_h$.
It satisfies the error estimate
\begin{equation*}
\|D^2u-D w_h\|_{L^2(\Omega)}
\leq
C(\lambda,\gamma,\varepsilon,\tau)
\inf_{v_h\in W_h}
\|D^2u-D v_h\|_{L^2(\Omega)}
\end{equation*}
where
\begin{equation*}
C(\lambda,\gamma,\varepsilon,\tau)
=
2 
\frac{
\sqrt{\|A\|_{L^\infty(\Omega)}^2+\sigma_\lambda^\tau(\gamma,\varepsilon)^2}
}
{c_\lambda^\tau(\gamma,\varepsilon)}
\left[\tilde\beta^{-1} + 
\frac{
\sqrt{\|A\|_{L^\infty(\Omega)}^2+\sigma_\lambda^\tau(\gamma,\varepsilon)^2}
}
{c_\lambda^\tau(\gamma,\varepsilon)}
\right].
\end{equation*}
Moreover the following reliable a~posteriori error estimate holds
\begin{equation*}
c_\lambda^\tau(\gamma,\varepsilon) \|D^2 u -Dw_h \|_{L^2(\Omega)}
\leq 
\sqrt{\| A:D w_h -f \|_{L^2(\Omega)}^2
+
\sigma_\lambda^\tau(\gamma,\varepsilon)^2
\|\rot w_h \|_{L^2(\Omega)}^2}
.
\end{equation*}
For any open subdomain $\omega\subseteq \Omega$ we have the efficiency
\begin{equation*}
\begin{aligned}
&
\sqrt{
\| A:D w_h -f \|_{L^2(\omega)}^2
+\sigma_\lambda^\tau(\gamma,\varepsilon)^2 \|\rot w_h \|_{L^2(\omega)}^2
}
\\
&\qquad\qquad\qquad\qquad
\leq 
\sqrt{\|A\|_{L^\infty(\omega)}^2 
+ \sigma_\lambda^\tau(\gamma,\varepsilon)^2}
\|D^2 u -Dw_h \|_{L^2(\omega)}.
\end{aligned}
\end{equation*}
\end{proposition}
\begin{proof}
The estimate \eqref{e:ABpointwise} and Lemma~\ref{l:sueli},
the relation \eqref{e:divrotestimate} and the triangle inequality
prove for any $v\in W$ that
\begin{equation*}
\begin{aligned}
\|\ddiv v\|_{L^2(\Omega)}
&\leq
\|\gamma A:Dv\|_{L^2(\Omega)} + \sqrt{1-\varepsilon} 
\|Dv\|_{L^2(\Omega)}
\\
&\leq
\|\gamma\|_{L^\infty(\Omega)}
\|A:Dv\|_{L^2(\Omega)} + \sqrt{1-\varepsilon} \,
(\|\ddiv v\|_{L^2(\Omega)} + \|\rot v\|_{L^2(\Omega)}) .
\end{aligned}
\end{equation*}
With the constant $c(\gamma,\varepsilon)$, this is
equivalent to
\begin{equation*}
c(\gamma,\varepsilon) \|\ddiv v\|_{L^2(\Omega)}
\leq
\|A:Dv\|_{L^2(\Omega)}
+
\sqrt{1-\varepsilon}/\|\gamma\|_{L^\infty(\Omega)}
\|\rot v\|_{L^2(\Omega)} .
\end{equation*}
Taking squares on both sides and using the Young inequality,
one arrives at
\begin{equation*}
\begin{aligned}
&
c(\gamma,\varepsilon)^2 \|\ddiv v\|_{L^2(\Omega)}^2
\\
&\quad
\leq
(1+\lambda)
\|A:Dv\|_{L^2(\Omega)}^2
+
(1+1/\lambda)
(1-\varepsilon)/\|\gamma\|_{L^\infty(\Omega)}^2
\|\rot v\|_{L^2(\Omega)}^2 .
\end{aligned}
\end{equation*}
Adding 
$c(\gamma,\varepsilon)^2 \|\rot v\|_{L^2(\Omega)}^2$
and dividing by $(1+\lambda)$
leads with \eqref{e:divrotestimate} to
\begin{equation*}
c_\lambda^\tau(\gamma,\varepsilon)^2 \|D v\|_{L^2(\Omega)}^2
\leq
\|A:Dv\|_{L^2(\Omega)}^2
+
\sigma_\lambda^\tau(\gamma,\varepsilon)^2
\|\rot v\|_{L^2(\Omega)}^2 .
\end{equation*}
Hence, $\tilde a$ satisfies the coercivity
\begin{equation*}
c_\lambda^\tau(\gamma,\varepsilon)^2 \|D v\|_{L^2(\Omega)}^2
\leq
\tilde a(v,v)
\quad\text{for any }v\in W.
\end{equation*}
As in the proof of Proposition~\ref{p:mixedanalysisNS},
this and the stability condition \eqref{e:infsup_d} establish
the unique solvability and the a~priori error estimate
with the constant from
\cite[Thm.~5.2.2]{BoffiBrezziFortin2013}.
The proof of the  a~posteriori bounds is immediate.
\end{proof}

\begin{remark}
The a~posteriori bounds in 
Propositions~\ref{p:confanalysis},
\ref{p:mixedanalysisNS}, \ref{p:mixedanalysisLS}
are fully explicit.
The evaluation of integrals of the $L^\infty$ coefficient $A$
may, however, be computationally challenging in practice,
cf.\ the numerical experiments in \S\ref{s:num}.
\end{remark}

\begin{remark}
Clearly,
an a~priori error estimate for the difference
$p-p_h$ in the $L^2$ norm in the fashion of
Propositions~\ref{p:mixedanalysisNS}, \ref{p:mixedanalysisLS}
can also be obtained.
Since, in this context, the
Lagrange multiplier is not of particular interest,
its analysis not included in the proposition.
Moreover, using \eqref{e:infsup_d} and the fact that $p=0$,
it can be shown that the error is bounded from above
by some constant times the suggested error estimator.
\end{remark}

\begin{remark}
In three space dimensions, subspaces of $Q$ must satisfy 
a pointwise divergence-free constraint.
In \cite{Gallistl2016}
the space $Q_h$
of divergence-free lowest-order Raviart-Thomas fields
was used. This space 
$Q_h$ consists exactly of all piecewise constant
vector fields that are continuous in the inter-element normal
directions
and whose normal component vanishes on the boundary $\partial\Omega$.
Another approach could be to further soften the formulation by
enforcing the divergence-free constraint in a weak manner.
This would involve an additional Lagrange multiplier
also arising in the error estimates.
\end{remark}

\section{An adaptive algorithm and its convergence}\label{s:afem}

This section is devoted to the description of an adaptive algorithm
for the  discretization methods from \S\ref{s:FEM}
and the proof of its convergence.
\subsection{Assumptions on the discrete spaces}
Let $\mathbb T$ denote a set of admissible shape-regular partitions
refined from some initial mesh $\mathcal T_0$
of $\overline\Omega$. The partitions may consist of triangles/tetrahedra
or quadrilaterals/hexahedra.
Shape-regularity is meant in the sense
that (i) there exist positive constants $c$ and $C$ such
that for any $\mathcal T\in\mathbb T$ and any $T\in\mathcal T$,
$c\operatorname{meas}(T)\leq \diam(T)^d\leq C \operatorname{meas}(T)$
and (ii) any two neighbouring elements $T,K\in\mathcal T$ satisfy
$c\leq \diam(T)/\diam(K)$.
This property is respected by many refinement routines like
newest-vertex bisection \cite{BinevDahmenDeVore2004},
but also refinements involving hanging
nodes are allowed as long as the number of hanging nodes
per interface stays uniformly bounded.
The shape-regularity implies that there is some $\alpha>0$
such that any $T\in\mathcal T\in\mathbb T$ and any refined element
$\hat T\subsetneq T$ in a refined partition $\widehat{\mathcal T}$,
satisfy $\operatorname{meas}(\hat T)\leq \alpha \operatorname{meas}(T)$.
The discretization spaces from \S\ref{ss:standardfem}
(resp.\ \S\ref{ss:mixedFEM}) are
labelled with the partitions in $\mathbb T$ and are 
denoted by $V(\mathcal T)$ (resp.\ $W(\mathcal T)$ and $Q(\mathcal T)$)
rather than $V_h$ etc.\ in \S\ref{s:FEM}.
The spaces are assumed to be nested on refined triangulations.
It is assumed that $\mathbb T$ contains sufficiently many refinements
so that for any $\mathcal T\in\mathbb T$ and any $\delta>0$ there is 
some refinement $\widehat{\mathcal T}\in\mathbb T$ such that
for any $T\in\widehat{\mathcal T}$ the diameter satisfies
$\diam(T)\leq\delta$.
It is furthermore assumed that there exist a stable, projective,
quasi-local quasi-interpolation operator, i.e.,
there is a constant $C$ such that for any $\mathcal T\in\mathbb T$
there is a linear idempotent map $I_{\mathcal T}:V\to V(\mathcal T)$
(resp. $I_{\mathcal T}:W\to W(\mathcal T)$ for the mixed method)
such that, for any $T\in\mathcal T$,
the estimate
\begin{align}
 \notag
&
\| D^2 I_{\mathcal T}v\|_{L^2(T)} \leq C \| D^2 v\|_{L^2(\omega_T)}
\quad\text{for any }v\in V \quad\text{(resp.\ }
\\
\label{e:clementMixed}
&
\diam(T)^{-1} \| z- I_{\mathcal T} z\|_{L^2(T)} +
\| D I_{\mathcal T}z\|_{L^2(T)} \leq C \| D z\|_{L^2(\omega_T)}
\quad\text{ for any }z\in W \text{)}
\end{align}
holds, where $\omega_T$ denotes the element-patch of $T$, i.e.,
the union of all elements of $\mathcal T$ sharing a point with $T$.
(This assumption can be relaxed by requiring $\omega_T$ to be
some surrounding domain with finite overlap property.)
Since this quasi-interpolation is a stable projection, it is also
quasi-optimal. It is assumed that for any sequence 
$(\mathcal T_\ell)_\ell$ of partitions with
$\max_{T\in\mathcal T_\ell}\diam(T)\to 0$ as $\ell\to\infty$,
the spaces $V(\mathcal T_\ell)$ 
(resp.\ $W(\mathcal T_\ell)$) are dense in $V$ (resp.\ $W$).
This implies that, for any $v\in V$,
the quasi-interpolation $I_{\mathcal T_\ell}v$
converges to $v$ in the $H^2$ norm (resp.\ in the $H^1$ norm).
These requirements are met for most of the known $H^2$ conforming
finite elements based on piecewise polynomials.
It is, however, important to note that not all
$H^2$ conforming finite elements lead to nested spaces.
The Argyris FEM and the Hsieh--Clough--Tocher FEM \cite{Ciarlet1978},
for example, do not satisfy this property.
A positive example is the BFS FEM \cite{Ciarlet1978}
used in the numerical experiments below.
In the case of mixed methods, the discretizations of $W$ need only
be $H^1$ conforming, and quasi-interpolation operators for such
spaces are well-established.
Their existence is typically assured
through the shape-regularity.
In addition,
the mixed finite element spaces are assumed to satisfy
Assumption~\ref{a:infsuploc} stated below.

\subsection{Adaptive algorithm and convergence proof}
The algorithm departs
 from an initial mesh $\mathcal T_0$ and runs the following
loop over the index $\ell=0,1,2,\dots$
\\
\textbf{Solve.}
Solve the discrete problem 
\eqref{e:discrprobconf} (resp.\ \eqref{e:discrsaddlepoint})
with respect to the mesh $\mathcal T_\ell$ and the
space $V(\mathcal T_\ell)$. Denote the solution by $u_\ell$
(resp.\ $(w_\ell,p_\ell))$.
\\
\textbf{Estimate.}
 Compute, for any $T\in\mathcal T_\ell$,
 the local error estimator contributions 
 $\eta^2_\ell(T) = \| A:D^2 u_\ell -f\|_{L^2(T)}^2$
 (resp.\ $\eta^2_\ell(T) = \| A:D w_\ell -f\|_{L^2(T)}^2
              + \sigma_\lambda^\tau(\gamma,\varepsilon)^2 \| \rot w_\ell\|_{L^2(T)}^2$).
\\
\textbf{Mark.}
Choose some (any) subset $\mathcal M_\ell\subseteq\mathcal T_\ell$
satisfying $T\in\mathcal M_\ell$ for at least one
$T\in\mathcal T_\ell$ with 
$\eta_\ell^2(T) = \max_{K\in\mathcal T_\ell}\eta_\ell^2(K)$.
\\
\textbf{Refine.}
Compute a refined admissible partition $\mathcal T_{\ell+1}$
of $\mathcal T_\ell$ such that at least 
all elements of $\mathcal M_\ell$ are
refined.

\begin{remark}
In view of the different weights in the error 
estimators of
Propositions~\ref{p:mixedanalysisNS}--\ref{p:mixedanalysisLS},
 it is worth mentioning that
one can choose different weights for the contributions of 
$\eta_\ell(T)$. This is, however, of minor importance for the 
convergence analysis.
\end{remark}

This adaptive algorithm is formulated in a fairly general
way, see also \cite{MorinSiebertVeeser2008};
it admits various
existing marking procedures, for instance the maximum
marking or the D\"orfler marking \cite{Doerfler1996}.
The refinement step typically involves some minimality condition
on the refined partition to gain efficiency.
In \S\ref{s:num} an instance of an adaptive algorithm with more
details is presented.
However, the present form of the algorithm suffices for
convergence. A similar argument was used by
\cite{MorinSiebertVeeser2008}.
The difference is that the residuals in the present error estimator
are strong $L^2$ residuals. This has the effect that no data
oscillations enter the convergence analysis. The main reason is
that the efficiency proof does not require the usual
techniques employing bubble functions \cite{Verfuerth2013}.

Consider the sequence $(\mathcal T_\ell)_\ell$ produced by the
adaptive algorithm.
The convergence proofs employ the subset
$\mathcal K\subseteq \cup_{\ell\geq0}\mathcal T_\ell$
of never refined elements defined by
$$
\mathcal K :=\bigcup_{\ell\geq0}\bigcap_{m\geq\ell}\mathcal T_m.
$$
The set was already utilized in \cite{MorinSiebertVeeser2008}.
It is the set of elements that are never refined once they are
created.
Accordingly, for any $\ell\geq0$, the partition $\mathcal T_\ell$
can be written as the following disjoint union
\begin{equation}\label{e:TellDecomp}
 \mathcal T_\ell =
    \mathcal K_\ell \cup \mathcal R_\ell
\qquad\text{for }
\mathcal K_\ell:=\mathcal K\cap\mathcal T_\ell
\text{ and }
\mathcal R_\ell:=\mathcal T_\ell\setminus\mathcal K_\ell .
\end{equation}
By definition, every element of $\mathcal R_\ell$
is eventually refined and the measure of those children that do not belong 
to $\mathcal K$ is reduced by a factor $\alpha$.
Hence, for any $\varepsilon>0$ there exists some
$\ell_0\geq0$ such that, for all $\ell\geq\ell_0$,
\begin{equation}\label{e:maxmeaseps}
 \max_{T\in\mathcal R_\ell}\operatorname{meas}(T)<\varepsilon.
\end{equation}
The following result proves the convergence of the adaptive
conforming scheme.
More details on the arguments employed here can be found in, e.g.,
\cite{MorinSiebertVeeser2008,NochettoSiebertVeeser2009}.

\begin{proposition}\label{p:afemconvergenceConforming}
Let $\tau=\tau\ns$ or $\tau=\tau\ls$.
 Let the admissible partitions $\mathbb T$ and the discrete
spaces $V(\mathcal T)$ for any $\mathcal T\in\mathbb T$ satisfy
the assumptions from the beginning of this section.
Then the sequence $u_\ell\in V(\mathcal T_\ell)$
produced by the adaptive algorithm
converges to the exact solution $u\in V$, i.e.,
$\| D^2(u-u_\ell)\|_{L^2(\Omega)}\to 0$ as $\ell\to\infty$.
\end{proposition}
\begin{proof}
The sequence $(u_\ell)_\ell$ converges in $V$ to some limit
$u_\star$ (because the $u_\ell$ are in particular the Galerkin
approximations of the variational problem posed on the closure
of the union of the nested spaces 
$V(\mathcal T_\ell)$, $\ell\geq 0$),
i.e.,
\begin{equation*}
 \|D^2(u_\star - u_\ell) \|_{L^2(\Omega)}\to 0
\quad\text{as }\ell\to\infty.
\end{equation*}
It is not a~priori known that $u_\star = u$, since it is not clear
whether the global mesh-size converges to zero in all regions of the 
domain $\Omega$.
Recall
decomposition \eqref{e:TellDecomp} of $\mathcal T_\ell$ in
never refined ($\mathcal K_\ell$) and eventually refined
($\mathcal R_\ell$) elements.
The local efficiency of the error estimator 
(Proposition~\ref{p:confanalysis})
and the
triangle inequality prove for any $T\in\mathcal T_\ell$ that
\begin{equation*} 
 \eta_\ell^2(T)
\leq 2\|A\|_{L^\infty(\Omega)}^2
\big( \| D^2(u-u_\star)\|_{L^2(T)}^2 
         + \| D^2(u_\star-u_\ell)\|_{L^2(T)}^2
\big).
\end{equation*}
Since 
$\| D^2(u_\star-u_\ell)\|_{L^2(T)}\to0$
 as $\ell\to\infty$, one concludes with
\eqref{e:maxmeaseps} that for any $\rho>0$ there exists some
$\tilde\ell_0\geq0$ such that, for all $\ell\geq\tilde\ell_0$,
$$
 \max_{T\in\mathcal K_\ell} \eta_\ell^2(T) \leq
 \max_{T\in\mathcal R_\ell} \eta_\ell^2(T) 
 <\rho
$$
(if this was not the case, the properties of the marking step
 would imply that eventually an element in $\mathcal K_\ell$
is refined, which contradicts its membership in $\mathcal K_\ell$).
In conclusion, for any $T\in\mathcal K$,
$\|A:D^2 u_\ell -f\|_{L^2(T)}\to0$ as $\ell\to\infty$.
In order to conclude convergence in $L^2(\cup\mathcal K)$,
observe that
$\sum_{T\in\mathcal K}\operatorname{meas}(T)
\leq \operatorname{meas}(\Omega)$,
that is, the series on the left-hand side converges.
Thus, the measure of $(\cup\mathcal K)\setminus (\cup\mathcal K_m)$
becomes arbitrary small for sufficiently large $m$.
Given $\delta>0$, the dominated convergence theorem therefore
implies that, for sufficiently large $m_0$ and
all $m\geq m_0$,
$\| A:D^2 u_\star - f\|_{L^2((\cup\mathcal K)\setminus(\cup\mathcal K_m)))}\leq \delta/2$.
The elementwise convergence
and $u_\ell\to u_\star$ in $H^2(\Omega)$ therefore imply that,
for some sufficiently large $\hat\ell_0>0$, there holds
\begin{equation*}
\begin{aligned}
 \| A:D^2 u_\ell - f\|_{L^2(\cup\mathcal K)}
&
\leq
 \| A:D^2 u_\ell - f\|_{L^2(\cup\mathcal K_{m_0})}
+
\| A:D^2 (u_\ell - u_\star)\|_{L^2((\cup\mathcal K)\setminus(\cup\mathcal K_{m_0}))}
\\
&\quad
+
\| A:D^2 u_\star - f\|_{L^2((\cup\mathcal K)\setminus(\cup\mathcal K_{m_0}))}
\leq \delta
\qquad\text{for all }\ell\geq\hat\ell_0.
\end{aligned}
\end{equation*}
Hence,
$\| A:D^2 u_\ell - f\|_{L^2(\cup\mathcal K)} \to 0$
as $\ell\to\infty$.

The $L^2$ identity $A:D^2u=f$,
the definition of $\tau$ and \eqref{e:lowerboundSmearsSueli} imply
\begin{equation*}
 \sum_{T\in\mathcal T_\ell} \eta_\ell^2(T)
\begin{cases}
 \leq 
  \frac{ \|A\|_{L^\infty(\Omega)}^2}
       {1-\sqrt{1-\varepsilon}}
  (A:D^2(u-u_\ell),\tau(\nabla(u-u_\ell)))_{L^2(\Omega)}
 &\text{if }\tau=\tau\ns,
\\
 =(A:D^2(u-u_\ell),\tau(\nabla(u-u_\ell)))_{L^2(\Omega)}
&\text{if }\tau=\tau\ls.
\end{cases}
\end{equation*}
Thus, with a constant $c>0$ (depending on the choice of $\tau$),
the Galerkin orthogonality, the Cauchy inequality and 
the $L^2$ identity $A:D^2u=f$  lead, 
for any $\ell\geq0$, to
\begin{equation*}
\begin{aligned}
 c
 \sum_{T\in\mathcal T_\ell} \eta_\ell^2(T)
&
\leq
(A:D^2(u-u_\ell),\tau(\nabla(u-u_\ell)))_{L^2(\Omega)}
\\
&=
 (A:D^2(u-u_\ell),\tau(\nabla(u-I_{\mathcal T_\ell}u)))_{L^2(\Omega)}
\\
&
\leq
\| A:D^2 u_\ell -f\|_{L^2(\cup\mathcal K)}
\| \tau(\nabla(u-I_{\mathcal T_\ell}u))\|_{L^2(\cup\mathcal K)}
\\
&\qquad
+
\| A:D^2 u_\ell -f\|_{L^2(\cup\mathcal R_\ell)}
\| \tau(\nabla(u-I_{\mathcal T_\ell}u))\|_{L^2(\cup\mathcal R_\ell)}.
\end{aligned}
\end{equation*}
The convergence 
$A:D^2u_\ell\to f$ in $L^2(\cup\mathcal K)$
and the convergence of the
quasi-interpolation show that both summands on the right-hand
side converge to zero.
Since by Proposition~\ref{p:confanalysis} the error estimator is
reliable, the proof is concluded.
\end{proof}

The convergence proof of the mixed scheme requires a
mild structural hypothesis on the discrete spaces.
Let $\ell\in\mathbb N$ and $T\in\mathcal K_\ell$
and consider the first-order neighbourhood and its triangulation
defined by
$$
\omega_{T,\mathcal K}:=
\operatorname{int}(\cup\{K\in\mathcal K: T\cap K\neq \emptyset\})
\quad\text{with}\quad
 \mathcal T(\omega_{T,\mathcal K})
 :=
 \{K\in\mathcal K: T\cap K\neq \emptyset\}.
$$
The shape-regularity assumption assures that there exists
$n(\ell,T)\geq\ell$ such that
$\mathcal T(\omega_{T,\mathcal K})\subseteq\mathcal K_m$
holds
for all $m\geq n(\ell,T)$. This means that, eventually, all neighbours
of $T$ belong to $\mathcal K$.
Define 
the spaces $W(\mathcal T(\omega_{T,\mathcal K}))$ of functions
in $W(\mathcal T_m)$ with support in $\overline{\omega_{T,\mathcal K}}$
and $Q(\mathcal T(\omega_{T,\mathcal K}))$ of functions from 
$Q(\mathcal T_m)$ that are restricted to $\omega_{T,\mathcal K}$.
\begin{assumption}\label{a:infsuploc}
All $\ell\in\mathbb N$ and all $T\in\mathcal K_\ell$ satisfy:
if $q\in Q(\mathcal T(\omega_{T,\mathcal K}))$ fulfils
$(\rot v,q)_{L^2(\omega_{T,\mathcal K})}=0$
 for all $v\in W(\mathcal T(\omega_{T,\mathcal K}))$,
then $(\rot v,q)_{L^2(\omega_{T,\mathcal K})}=0$ holds for 
all $v\in H^1_0(\omega_{T,\mathcal K};\mathbb R^d)$.
\end{assumption}
Assumption~\ref{a:infsuploc}
basically states that a full-rank condition like \eqref{e:infsup_d}
remains true on element patches. It is, however, a purely algebraic
and, thus, weaker condition. As will turn out in the proof 
of Proposition~\ref{p:afemconvergenceMixed}, the assumption could be 
further relaxed by allowing other suitable overlapping neighbourhoods.
In two space dimensions, Assumption~\ref{a:infsuploc} 
(or a relaxed version with larger patches)
is satisfied
by many (if not all) known finite elements for the Stokes problem.
In particular, pairings whose stability proof is based on
the design of a Fortin operator or the macro-element technique 
are included.
Pathological situations in which an element $T$ does not have
enough neighbours
(e.g., all vertices on lie on $\partial \mathcal K$) can 
only occur if $T$ meets the boundary $\partial\Omega$
(due to the shape-regularity). Hence, such situations are excluded
if the class $\mathbb{T}$ of triangulations is chosen properly.
In three dimensions, where the problem significantly differs from the 
Stokes equations, the pairing from \cite{Gallistl2016}, which is based 
on face-bubble stabilization, satisfies Assumption~\ref{a:infsuploc}.
In all these cases, the verification of Assumption~\ref{a:infsuploc}
is independent of any a~priori knowledge about the sets 
$\mathcal K_\ell$.
The next result proves the convergence of the adaptive mixed
scheme.
\begin{proposition}\label{p:afemconvergenceMixed}
Let $\tau=\tau\ns$ or $\tau=\tau\ls$.
 Let the admissible partitions $\mathbb T$ and the discrete
spaces $W(\mathcal T)$, $Q(\mathcal T)$
 for any $\mathcal T\in\mathbb T$ satisfy
the assumptions from the beginning of this section
as well as Assumption~\ref{a:infsuploc}.
Then the sequence
$(w_\ell,p_\ell)\in W(\mathcal T_\ell)\times Q(\mathcal T_\ell)$
produced by the adaptive algorithm
converges to the exact solution $(w,p)\in W\times Q$. In particular,
$\| D^2 u- Dw_\ell\|_{L^2(\Omega)}\to 0$ as $\ell\to\infty$.
\end{proposition}
\begin{proof}
Similar as in the proof of Proposition~\ref{p:afemconvergenceConforming},
one can show that there exists a limit
$(w_\star,p_\star)\in W\times Q$ such that
$
\|D(w_\star - w_\ell) \|_{L^2(\Omega)}
+\|p_\star - p_\ell \|_{L^2(\Omega)}
\to 0
\quad\text{as }\ell\to\infty.
$
Consider again the set $\mathcal K$ and the decomposition
\eqref{e:TellDecomp}.
The local efficiency of the error estimator 
(Propositions~\ref{p:mixedanalysisNS} and \ref{p:mixedanalysisLS})
and the
triangle inequality prove for 
some constant $C>0$ (depending on the choice of $\tau$) and
any $T\in\mathcal T_\ell$ that
\begin{equation*} 
 \eta_\ell^2(T)
\leq
C
\big( \| D(w-w_\star)\|_{L^2(T)}^2 
         + \| D(w_\star-w_\ell)\|_{L^2(T)}^2
\big).
\end{equation*}
As in the proof of Proposition~\ref{p:afemconvergenceConforming},
one obtains
$$
\| A:D w_\ell - f\|_{L^2(\cup\mathcal K)}
  + \|\rot w_\ell\|_{L^2(\cup\mathcal K)} \to 0
 \quad\text{as }\ell\to\infty.
$$
Similar as in the proof of 
Proposition~\ref{p:afemconvergenceConforming},
the coercivity of $\tilde a_\tau$ shows that there exists a constant
$c>0$ (depending on the choice of $\tau$) such that with
the $L^2$ identity $A:D w=f$
and the quasi-interpolation $I_{\mathcal T_\ell}$ 
the following split is valid
\begin{equation}
\label{e:mixedsplit}
\begin{aligned}
&
 c\sum_{T\in\mathcal T_\ell} \eta_\ell^2(T)
\leq
(A:D(w-w_\ell), \tau(w-I_{\mathcal T_\ell}w))_{L^2(\Omega)}
\\
&\qquad
+(A:D(w-w_\ell), \tau(I_{\mathcal T_\ell}w -w_\ell))_{L^2(\Omega)}
+
\sigma_\lambda^\tau(\gamma,\varepsilon)^2
\|\rot w_\ell\|_{L^2(\Omega)}^2 .
\end{aligned}
\end{equation}
The first term on the right-hand side of \eqref{e:mixedsplit}
can be shown to converge to zero with the techniques from
Proposition~\ref{p:afemconvergenceConforming} because locally
it consists of products of error estimator and interpolation
error contributions.
Using the discrete equations 
\eqref{e:discrsaddlepoint_b}--\eqref{e:discrsaddlepoint_c}
and $A:Dw=f$,
the remaining terms of \eqref{e:mixedsplit} are transformed into
\begin{equation}\label{e:afemproofSplit} 
\sigma_\lambda^\tau(\gamma,\varepsilon)^2
(\rot w_\ell, \rot I_{\mathcal T_\ell}w)_{L^2(\Omega)} 
+
(\rot I_{\mathcal T_\ell}w, p_\ell)_{L^2(\cup \mathcal R_\ell)}
+
(\rot I_{\mathcal T_\ell}w, p_\ell)_{L^2(\cup \mathcal K_\ell)}
.
\end{equation}
Since $\rot I_{\mathcal T_\ell}w = \rot (I_{\mathcal T_\ell}w-w)$,
the first and second term can again be shown to converge to zero:
the first term is an
element-wise product of error estimator and interpolation error
contributions, while the second one is controlled by the 
interpolation error on the elements of $\mathcal R_\ell$.
The last term of \eqref{e:afemproofSplit} can be rewritten as
$
(\rot I_{\mathcal T_\ell}w, p_\ell-p_\star)_{L^2(\cup \mathcal K_\ell)}
+
(\rot I_{\mathcal T_\ell}w, p_\star)_{L^2(\cup \mathcal K_\ell)}
$
and, since $p_\ell\to p_\star$ in $L^2$, it remains to estimate
the term
$(\rot I_{\mathcal T_\ell}w, p_\star)_{L^2(\cup \mathcal K_\ell)}$.

For the analysis of this term, it is useful to note that,
for any connected component $\tilde\omega$ of 
$\cup\mathcal K \setminus \partial(\cup\mathcal K)$,
the function
$p_\star|_{\tilde\omega}$
 satisfies
$(\rot v,p_\star)_{L^2(\tilde\omega)}=0$ 
for all $v\in H^1_0(\Omega;\mathbb R^d)$ with support in $\tilde\omega$.
For the proof of this claim,
let 
$\omega\subseteq
         \cup\mathcal K_\ell \setminus \partial(\cup\mathcal K_\ell)$
be a connected component of 
$\cup\mathcal K_\ell \setminus \partial(\cup\mathcal K_\ell)$.
It is not difficult to see that $\omega$ is an open subset of 
one of the connected components of
$\cup\mathcal K \setminus \partial(\cup\mathcal K)$.
Let $\mathcal K_\ell(\omega)$ denote the set of all elements
of $\mathcal K_\ell$ whose interior has a non-empty intersection with
$\omega$.
It is easily verified that the 
limit $(w_\star,p_\star)$ satisfies
$
\tilde a_\tau (w_\star,v_\ell) + b(v_\ell, p_\star)
=
(f,\tau(v_\ell))_{L^2(\Omega)}
$
for all $\ell\geq 0$ and all $v_\ell\in W(\mathcal T_\ell)$
supported in $\omega$.
The fact that $A:D w_\star = f$ 
and $\rot w_\star=0$
in the $L^2$ sense on 
$\cup\mathcal K$
show that $(\rot v_\ell,p_\star)_{L^2(\Omega)}=0$.
Assumption~\ref{a:infsuploc} and an overlap argument prove
that $(\rot v,p_\star)_{L^2(\tilde\omega)}=0$ for any 
$v\in H^1_0(\tilde\omega;\mathbb R^d)$ with compact support inside
$\tilde\omega$, which implies the claim.

Let $\mu\in W^{1,\infty}(\omega)$ denote a positive cut-off function
with values in the interval $[0,1]$
taking the value $1$
on all elements of $\mathcal K_\ell(\omega)$
that do not meet the boundary
$\Gamma_\omega:=\partial(\cup\mathcal K_\ell)\setminus\partial\Omega$
and that vanishes on $\Gamma_\omega$,
and that satisfies, for some constant $C$ and for any element
$T\in \mathcal K_\ell(\omega)$
touching $\Gamma$ that
$
 \|\nabla\mu\|_{L^\infty(T)}\leq C \diam(T)^{-1} .
$
The boundary conditions of $\mu$,
the identity $\rot w =0$, and the product rule
lead to
\begin{equation*}
\begin{aligned}
&
 | (\rot (I_{\mathcal T_\ell} w),p_\star )_{L^2(\omega)} |
=
 |(\rot ((1-\mu)(w-I_{\mathcal T_\ell} w)),p_\star )_{L^2(\omega)} |
\\
\quad
&\leq
\sum_{\substack{ T\in \mathcal K_\ell(\omega) 
           \\ T\cap\Gamma_\omega\neq\emptyset}}
\left(
\|\rot(w-I_{\mathcal T_\ell} w)\|_{L^2(T)}
+ C \diam(T)^{-1}
\|w-I_{\mathcal T_\ell} w\|_{L^2(T)}
 \right) \|p_\star\|_{L^2(T)}.
\end{aligned}
\end{equation*}
Using the Cauchy inequality and \eqref{e:clementMixed},
this term can be shown to converge to zero
because by the shape-regularity the 
measure of the elements in $\mathcal T_\ell$
meeting the boundary
$ \partial (\cup\mathcal K_\ell)\setminus\partial\Omega$
converges to zero as $\ell\to\infty$.
In conclusion, the error estimator converges to zero,
and so does the error.
\end{proof}

\section{Numerical results}\label{s:num}
This section presents numerical experiments in
two space dimensions 
for the choice $\tau\ls$, that is the least-squares method.
\subsection{Numerical realization}
This subsection describes
the employed finite element methods and the used adaptive
algorithm. 

\subsubsection{Conforming scheme}
The $H^2$-conforming method used here is the 
Bog\-ner--Fox--Schmit (BFS) finite element \cite{Ciarlet1978}.
Let $\mathcal T$ be a rectangular partition of $\Omega$,
where one hanging node (that is a point shared by two or more rectangles
which is not vertex to all of them) per edge is allowed.
The finite element space $V_h$ is the subspace of $V$ consisting
of piecewise bicubic polynomials.
It is a second-order scheme with expected convergence of 
$\mathcal O(h^2)$ for $H^4$-regular solutions
on quasi-uniform meshes with maximal
mesh-size $h$. For the error in the $H^1$ and the $L^2$ norm,
the corresponding convergence order is 
$\mathcal O(h^3)$ and $\mathcal O(h^4)$, respectively.

\subsubsection{Mixed scheme}
As a mixed scheme, the Taylor--Hood finite element
\cite{BoffiBrezziFortin2013}
 is used.
For a regular triangulation of $\mathcal T$ of $\Omega$,
the space $W_h$ is the subspace of $W$ consisting of piecewise
quadratic polynomials while
$Q_h$ is the subspace of $Q$ consisting of piecewise affine and
globally continuous functions.
It is a second-order scheme with expected convergence of 
$\mathcal O(h^2)$ for $H^3$-regular solution $w$ (meaning that 
$u$ is $H^4$-regular)
on quasi-uniform meshes. 
For the error $\|w-\nabla u_h\|_{L^2(\Omega)}$,
the corresponding convergence order is 
$\mathcal O(h^3)$.
The computation of the primal variable $u_h$ is performed with
a standard finite element method based on piecewise quadratics.
The predicted convergence order in the $L^2$ norm is
$\mathcal O(h^3)$.

\subsubsection{Adaptive algorithm}

For any element $T\in\mathcal T$
the error estimators of Proposition~\ref{p:confanalysis}
and Proposition~\ref{p:mixedanalysisLS}
are abbreviated as follows
\begin{equation*}
\begin{aligned}
\eta_{\mathrm{conf}}^2(T)& = \| A:D^2 u_h - f\|_{L^2(T)}^2, \\
\eta_{\mathrm{mixed}}^2(T)& =
\| A:D w_h -f \|_{L^2(T)}^2
+\sigma_\lambda^\tau(\gamma,\varepsilon)^2 \|\rot w_h \|_{L^2(T)}^2.
\end{aligned}
\end{equation*}
Furthermore set
$
\eta_{\mathrm{conf}}
 :=\sqrt{\sum_{T\in\mathcal T}\eta_{\mathrm{conf}}^2(T)}
$, 
$
\eta_{\mathrm{mixed}}
 :=\sqrt{\sum_{T\in\mathcal T}\eta_{\mathrm{mixed}}^2(T)}
$.
The following adaptive algorithm is a concrete instance of
the procedure outlined in \S\ref{s:afem}.
It is based on the D\"orfler marking
\cite{Doerfler1996}
for some parameter $0<\theta\leq 1$.
In the following, $\eta_\ell$ refers to 
$\eta_{\mathrm{conf}}$ or $\eta_{\mathrm{mixed}}$,
depending on the used method.
Departing from an initial mesh $\mathcal T_0$ it runs the following
loop over the index $\ell=0,1,2,\dots$
\\
\textbf{Solve.}
Solve the discrete problem with respect to the mesh $\mathcal T_\ell$.
\\
\textbf{Estimate.}
 Compute the local error estimator contributions $\eta_\ell^2(T)$, 
 $T\in\mathcal T_\ell$, for the discrete solution.
\\
\textbf{Mark.}
Mark a minimal subset $\mathcal M\subseteq \mathcal T_\ell$
such that 
$\theta \eta_\ell^2\leq 
\sum_{T\in\mathcal M} \eta_\ell^2(T)$.
\\
\textbf{Refine.}
Compute a refined admissible partition $\mathcal T_{\ell+1}$
of $\mathcal T_\ell$
of minimal cardinality such that all elements of $\mathcal M$ are
refined.

For rectangular meshes, the local refinement splits every
rectangle in four congruent sub-rectangles
while further local refinements assure the property of only one
hanging node per edge.
On triangular meshes, newest-vertex bisection
\cite{BinevDahmenDeVore2004} is employed.

\subsection{Setup}
In all numerical experiments the domain is the square
$\Omega = (-1,1)^2$.
The parameter $\lambda$ for the stabilization
in the mixed scheme is chosen as $\lambda=1$.
All convergence history plots are logarithmically scaled.
The errors are plotted against the number of degrees of
freedom $\mathtt{ndof}$,
that is, the space dimension of $V_h$
resp.\ of $W_h\times Q_h$.
In the adaptive computation, the parameter $\theta$
is chosen $\theta=0.3$.
The coefficient $A$ reads
$$
A = \left[
        \begin{smallmatrix}
         2 & x_1 x_2/ ( |x_1|\, |x_2|) \\
         x_1 x_2/ ( |x_1|\, |x_2|) & 2
         \end{smallmatrix} 
     \right].
$$
The requirements of \S\ref{s:formulation} are met with
$\varepsilon = 3/5$,
$\|\gamma\|_{L^\infty(\Omega)}=2/5$, and  $\|A\|_{L^\infty(\Omega)}=2$,
so that
$c(\gamma,\varepsilon) = 5/2 -\sqrt{5/2}>0.91886$.
Three test cases are considered:

\subsection{Experiment~1}
In the first experiment
the known smooth solution
$u(x) = x_1 x_2 (1-\exp(1-|x_1|)(1-\exp(1-|x_2|)$
from \cite{SmearsSueli2013}
is considered.
The convergence history is displayed in 
Figure~\ref{f:convhistBFSsmooth} for the conforming BFS discretization
and in Figure~\ref{f:convhistTHsmooth} for the mixed Taylor--Hood
method. The convergence rates are of optimal order, that is
$\mathcal O(\mathtt{ndof}^{-1})$ for the approximation of the
Hessian and
$\mathcal O(\mathtt{ndof}^{-3/2})$ for the approximation of the 
gradient.
With the BFS element, $u$ is approximated  at the optimal
rate $\mathcal O(\mathtt{ndof}^{-2})$.
The mixed method gives the rate
$\mathcal O(\mathtt{ndof}^{-3/2})$, which is optimal for the used
quadratic FEM.
Since the solution in this example is smooth and the discontinuities
of the coefficient match with the initial meshes, uniform refinement
leads to the same rates as adaptive refinement.
In the case of a non-matching initial triangulation, uniform mesh
refinement leads to reduced convergence rates as shown
in Figure~\ref{f:convhistSmooth_nonmatch}, whereas the adaptive
BFS and Taylor-Hood schemes seem to behave optimal.
The initial rectangular mesh is the square subdivided in four
rectangles meeting at $(0.1,0.2)$. The initial triangular mesh
is created by inserting a `criss' diagonal in each of those
four rectangles.
A more challenging example is given below.

\begin{figure}
\centering
\includegraphics[width=.7\textwidth]{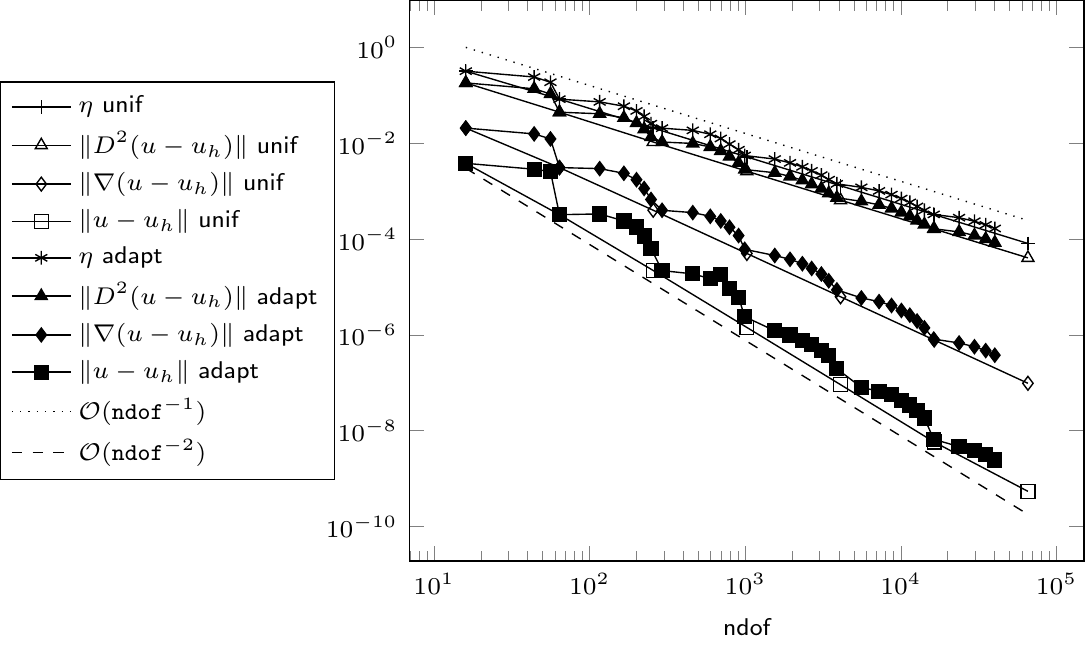}
\caption{Convergence history in the smooth Experiment~1 for the
   BFS finite element.\label{f:convhistBFSsmooth}}
\end{figure}

\begin{figure}
\centering
\includegraphics[width=.7\textwidth]{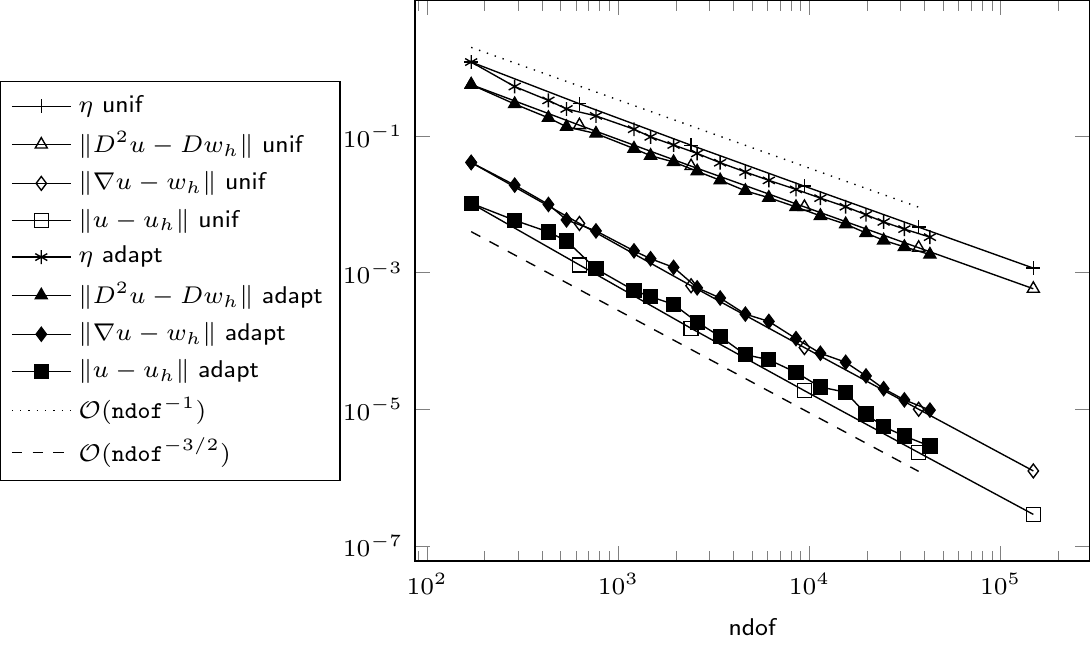}
\caption{Convergence history in the smooth Experiment~1 for the
   Taylor--Hood finite element.\label{f:convhistTHsmooth}}
\end{figure}

\begin{figure}
\centering
\includegraphics[width=.48\textwidth]{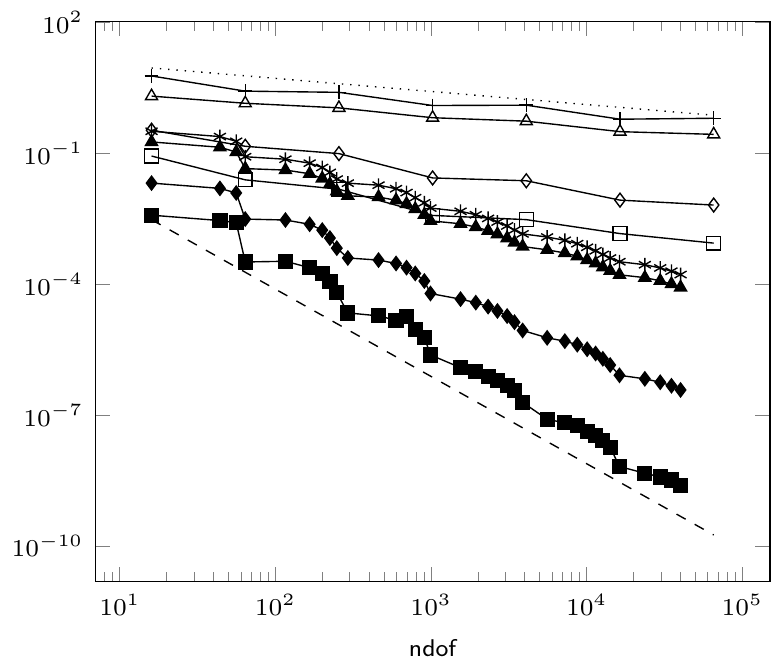}
\includegraphics[width=.48\textwidth]{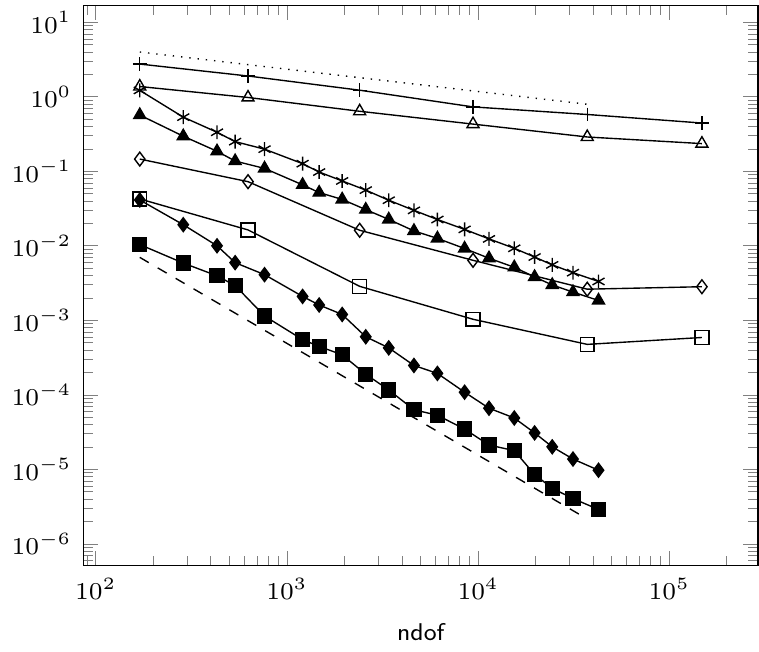}
\caption{Convergence history in the smooth Experiment~1
   with non-matching initial mesh. Left: BFS method, cf.\
   Figure~\ref{f:convhistBFSsmooth} for a legend.
   Right: BFS method, cf.\
   Figure~\ref{f:convhistTHsmooth} for a legend.
   In both plots, the dotted line indicates
   $\mathcal O (\mathtt{ndof}^{-0.3})$ (unlike in
   Figures~\ref{f:convhistBFSsmooth}--\ref{f:convhistTHsmooth}).
   \label{f:convhistSmooth_nonmatch}}
\end{figure}

\subsection{Experiment~2}
The known singular solution reads in polar coordinates as
$$
 u(r,\theta) =\begin{cases}
               r^{5/3} (1 - r)^{5/2} \sin(2\theta/3)^{5/2}
                   &\text{if } 0<r\leq 1 \text{ and } 0<\theta<3\pi/2\\
               0 &\text{else.}
            \end{cases}
$$
The Sobolev smoothness of $u$ near the origin $(0,0)$ is
strictly less than $H^{8/3}$.
Also near the
boundary of the sector $\{(r,\phi):0<r<1, 0<\theta<3\pi/2\}$
the regularity is reduced.
%
The singularities in Experiment~2 lead to the suboptimal convergence
rates of uniform refinement, displayed in the convergence
history of Figure~\ref{f:convhistBFSsingularity} for the BFS FEM
and Figure~\ref{f:convhistTHsingularity} for the Taylor--Hood
element. In both cases, the adaptive method converges at
optimal rate.
The graph of the solution computed with the adaptive
BFS element and the
adaptively generated meshes are displayed 
in Figure~\ref{f:meshessingularity}.
In both cases, the refinement is pronounced in the regions where
the solution is singular: the origin and the curved sector boundary.

\begin{figure}
\centering
\includegraphics[width=.7\textwidth]{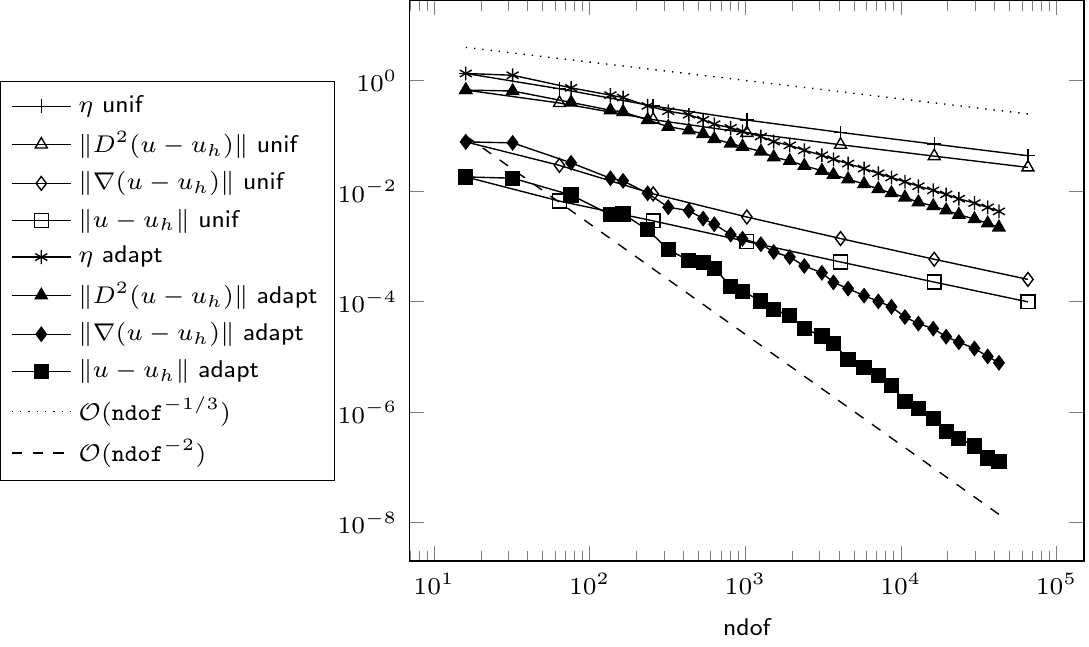}
\caption{Convergence history in the singular Experiment~2 for the
   BFS finite element.\label{f:convhistBFSsingularity}}
\end{figure}

\begin{figure}
\centering
\includegraphics[width=.7\textwidth]{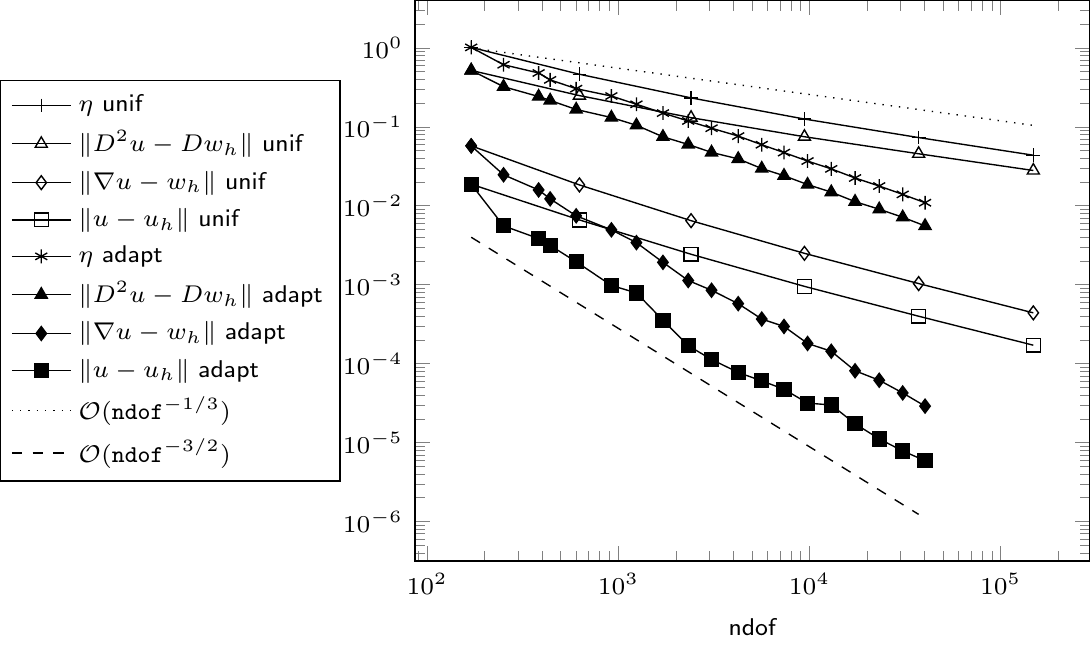}
\caption{Convergence history in the singular Experiment~2 for the
   Taylor--Hood finite element.\label{f:convhistTHsingularity}}
\end{figure}

\begin{figure}
\centering
\includegraphics[width=.33\textwidth]
            {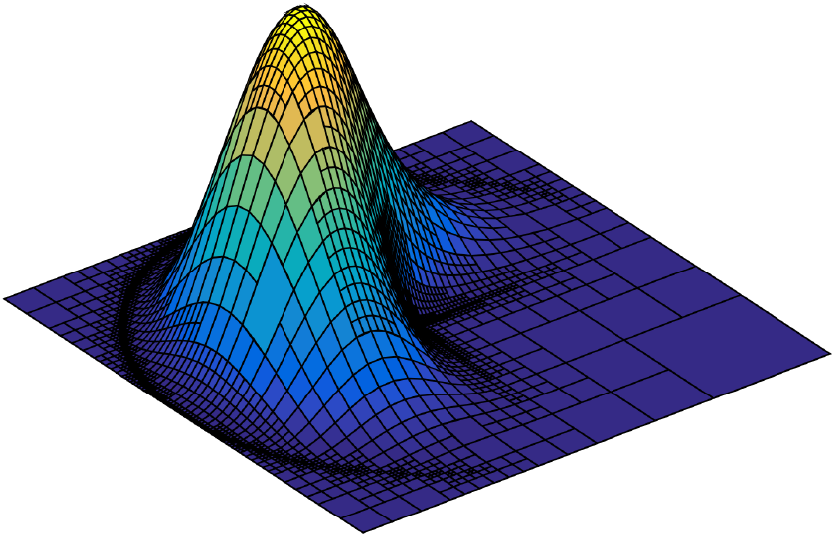}
\hfil
\includegraphics[width=0.23\textwidth]
            {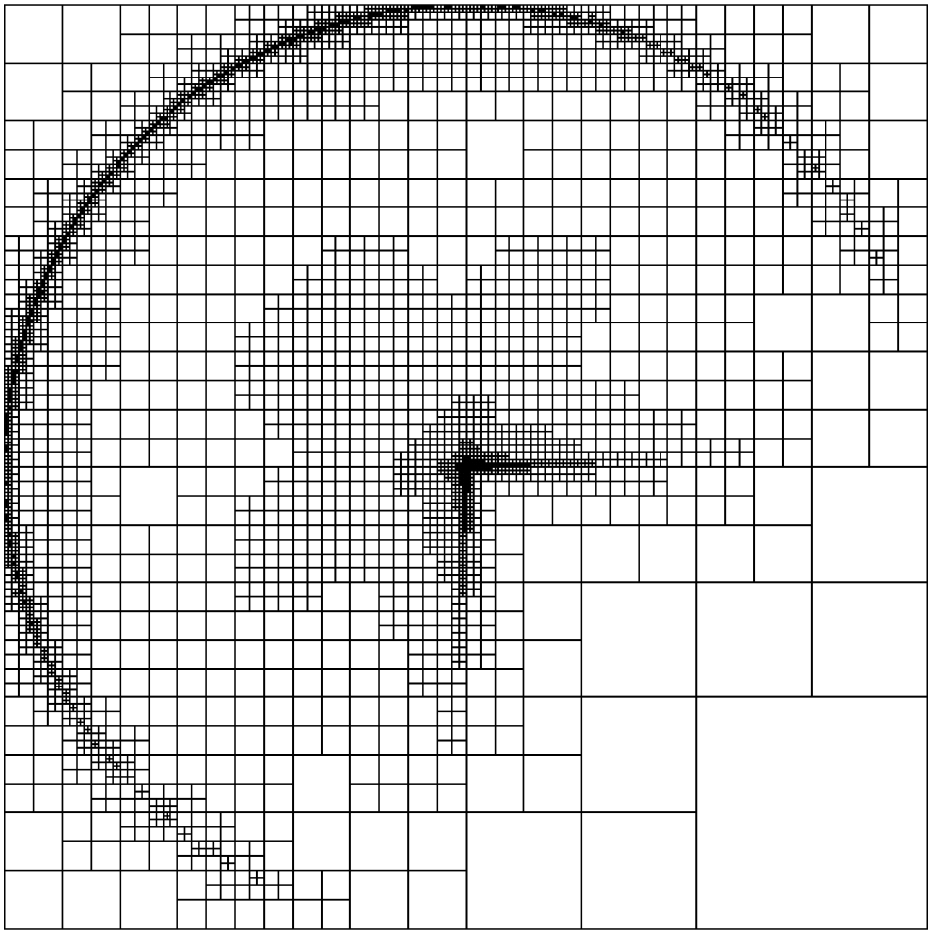}
\hfil
\includegraphics[width=0.23\textwidth]
            {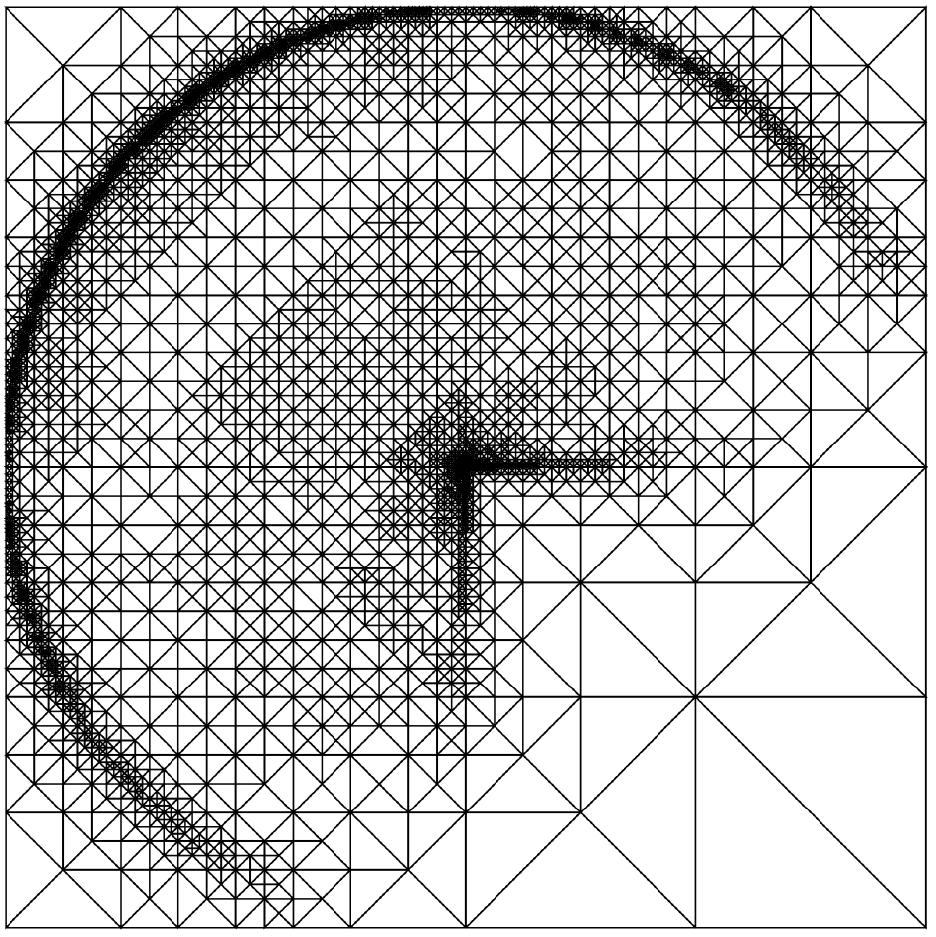}            
\caption{Experiment~2.
Surface plot of the discrete BFS solution (left);
and adaptive meshes.
Middle: BFS, 6\,017 vertices, 23\,584 degrees of freedom, 
$\ell=26$. Right: Taylor--Hood,
4\,518 vertices, 40\,238 degrees of freedom, $\ell=18$.%
\label{f:meshessingularity}}
\end{figure}

\begin{figure}
\centering
\includegraphics[width=.7\textwidth]{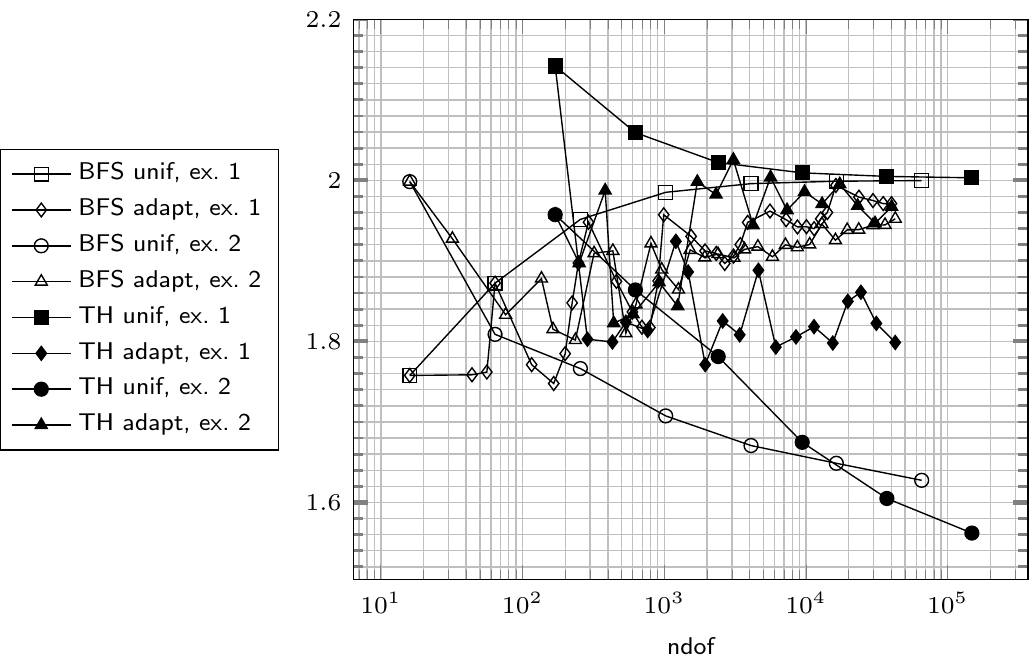}
\caption{Efficiency indices (estimator/error)
for the error estimators in 
Experiments~1--2 (BFS=Bogner--Fox--Schmit, TH=Taylor--Hood)
with matching initial meshes.
\label{f:effind}}
\end{figure}

\subsection{Efficiency indices for Examples~1--2}
The efficiency indices are defined by 
$\eta_{\mathrm{conf}} / \|D^2(u-u_h) \|_{L^2(\Omega)}$
for the conforming discretization and by
$\eta_{\mathrm{mixed}} / \|D^2 u-Dw_h \|_{L^2(\Omega)}$.
Propositions~\ref{p:confanalysis} and
\ref{p:mixedanalysisLS} and the values of $\|A\|_{L^\infty(\Omega)}$
and $c(\gamma,\varepsilon)$,
$c_\lambda^\tau(\gamma,\varepsilon)$,
and $\sigma_\lambda^\tau(\gamma,\varepsilon)$,
predict that the efficiency index
ranges in the interval
$[0.91886,2]$ for the conforming discretization
and in $[0.45943,2.2186]$ for the mixed method.
The efficiency indices for Experiments~1--2 
with matching initial meshes
are shown in 
Figure~\ref{f:effind}.
For the conforming discretization they range from $1.6$ to $2$,
while for the mixed scheme they lie between $1.5$ and $2.2$.

\subsection{Experiment 3}
This is an example with right-hand side
$f = 1$, where the exact solution is unknown.
The used coefficient is $A\circ\varphi$, that is, $A$ is 
concatenated with the nonlinear transform 
$\varphi
(x_1,x_2)
=
(x_1+1/3 \; , \; x_2-1/3 + (x_1+1/3)^{1/3})
$.
The coefficient is not aligned with the initial meshes and has
a sharp discontinuity interface near the point $(-1/3,-1/3+(1/3)^{1/3})$.
Figure~\ref{f:meshesnonmatching} displays the sign pattern of
its off-diagonal entries.
In Experiment~3, the meshes are not aligned with the discontinuous
coefficient. 
The convergence history is shown in
Figure~\ref{f:convhistBFS-THnonmatching} for the BFS method 
and the Taylor--Hood
method.
Since the exact solution is not known, the error estimators 
are plotted.
In both cases, 
uniform refinement leads to the suboptimal convergence
rate of 
$\mathcal O(\mathtt{ndof}^{-0.35})$.
The adaptive methods converge at a better rate. Still, it is 
suboptimal of rate $\mathcal O(\mathtt{ndof}^{-1/2})$.
This may be due to under-integration. Indeed, a Gaussian
quadrature rule is used, which is not accurate for discontinuous
function, 
and the adaptive method behaves like a first-order scheme.
The adaptive meshes from Figure~\ref{f:meshesnonmatching} show
strong refinement towards the jump of the coefficient.

\begin{figure}
\centering
\includegraphics[width=.7\textwidth]{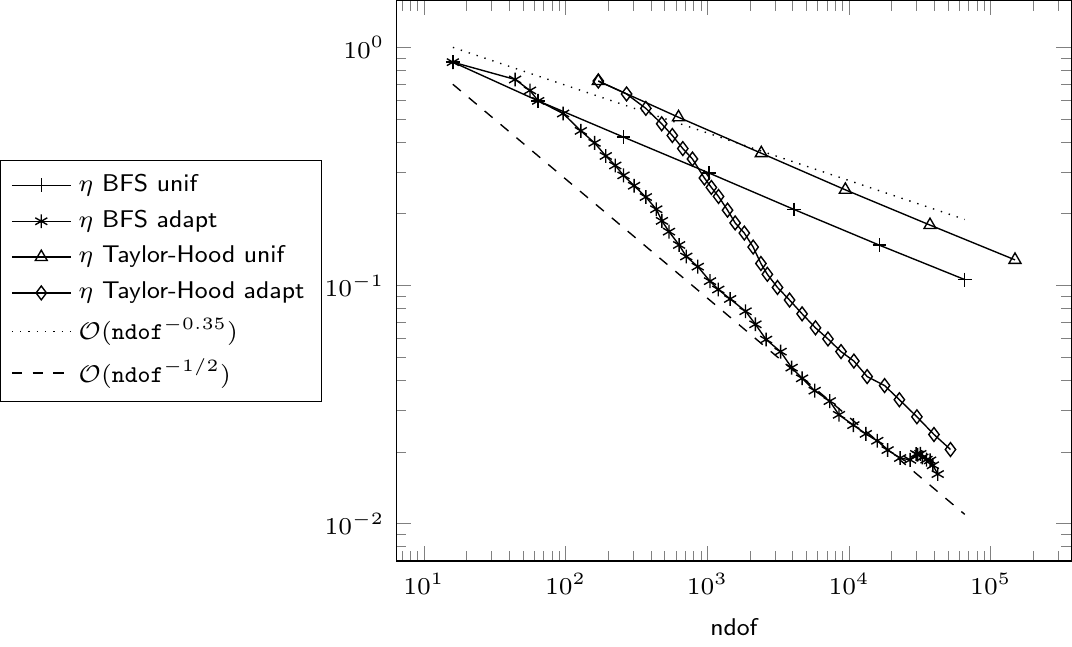}
\caption{%
Convergence history for the BFS and the Taylor--Hood FEM for
Experiment~3 with the nonmatching coefficient.
\label{f:convhistBFS-THnonmatching}}
\end{figure}

\begin{figure}
\centering
\includegraphics[width=0.23\textwidth]{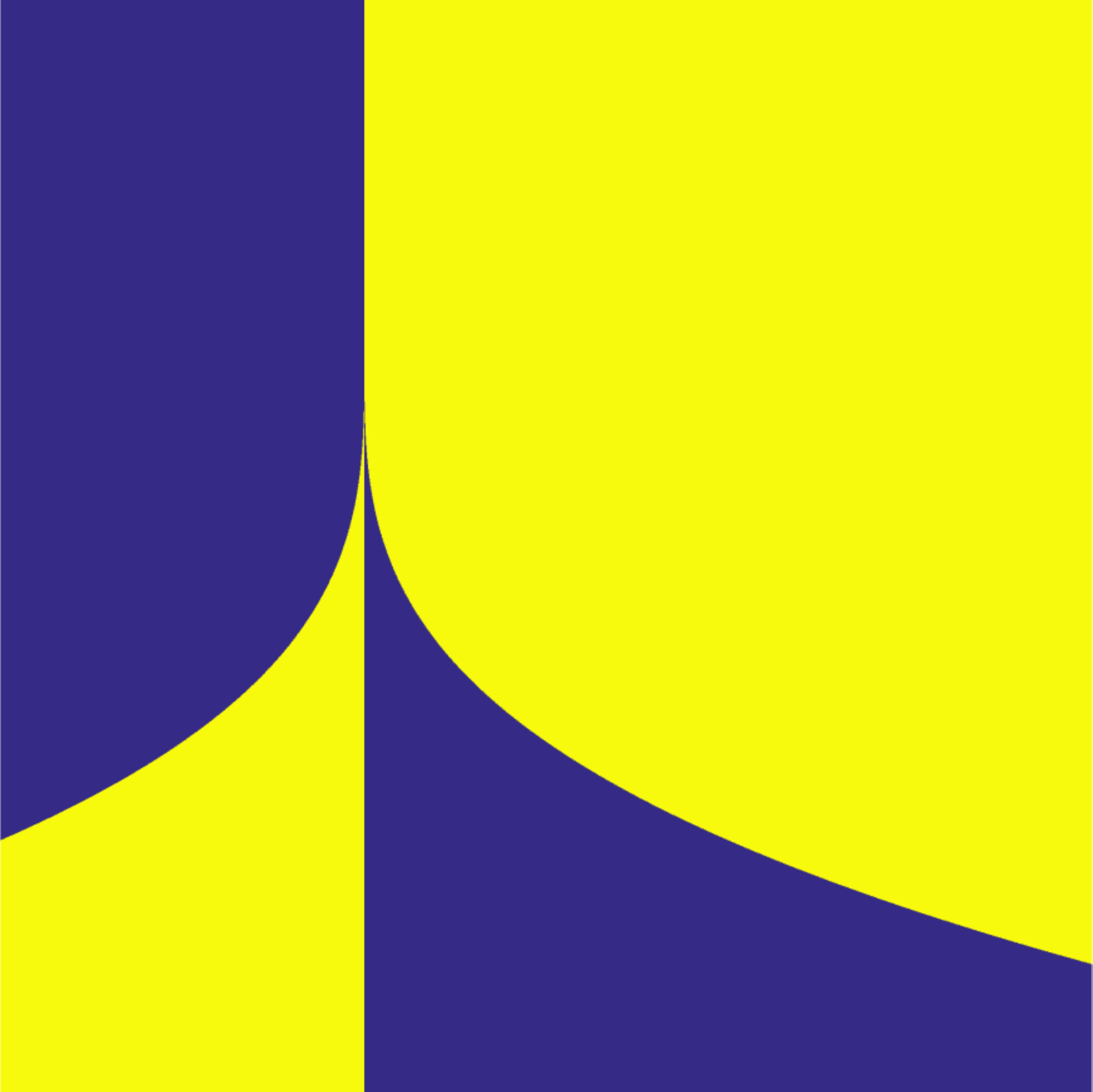}
\hfil
\includegraphics[width=0.23\textwidth]
            {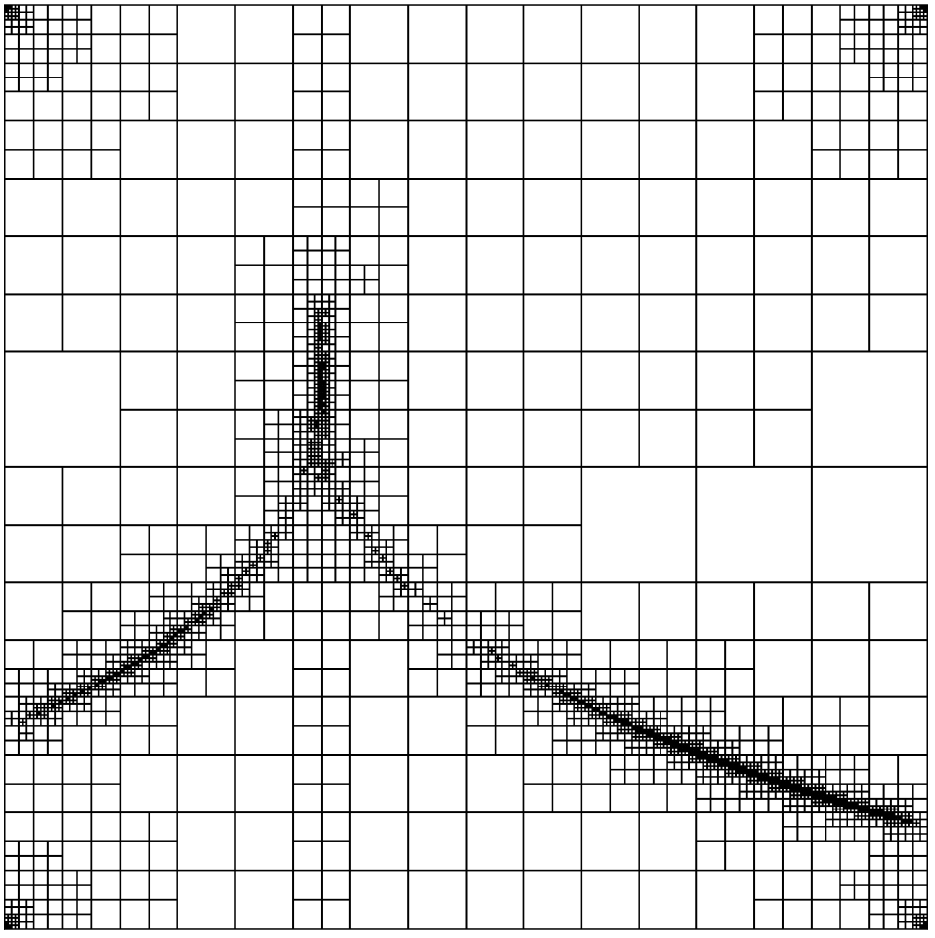}
\hfil
\includegraphics[width=0.23\textwidth]
            {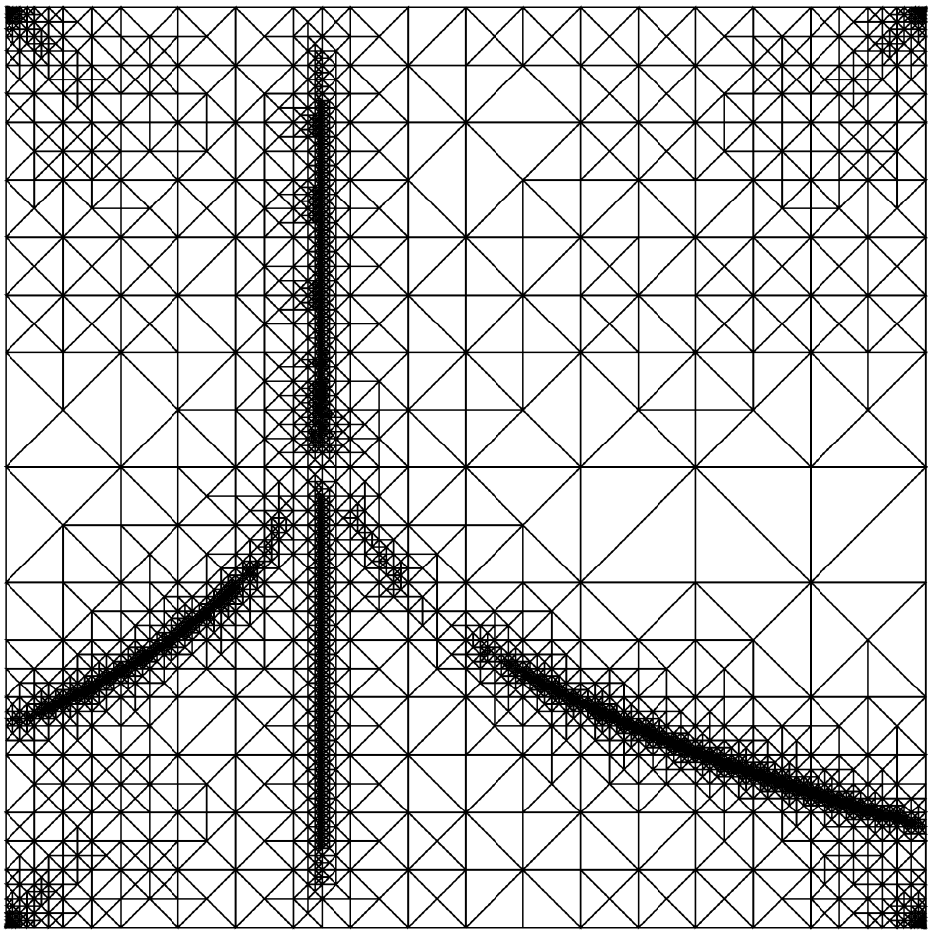}            
\caption{%
Experiment~3 with the nonmatching coefficient.
Sign pattern of the off-diagonal entries of the
coefficient $A$ (left); and
adaptive meshes. Middle: BFS, 4\,808 vertices, 18\,716 degrees of
freedom, $\ell=33$. Right: Taylor--Hood, 5\,848 vertices,
51\,956 degrees of freedom, $\ell=28$.
\label{f:meshesnonmatching}}
\end{figure}

\section{Conclusive remarks} \label{s:conclusion}

The 
variational formulation of \cite{SmearsSueli2013}
as well as the new least-squares formulation of elliptic equations in
nondivergence form can be discretized with 
conforming and, more importantly, mixed finite element
technologies in a direct way. This allows for quasi-optimal error
estimates and a~posteriori error analysis.
The proven convergence of the adaptive algorithm
can be observed in the numerical experiments,
and, as an empirical observation,
appears to be quasi-optimal, provided the
quadrature is accurate enough.
The following remarks conclude this paper.

\paragraph{(a) On the choice of the variational formulation}
The least-squares method is presented as an alternative approach
to the nonsymmetric formulation of \cite{SmearsSueli2013}.
While the symmetry of the discrete problem is certainly a favourable
property, a straightforward generalization to, e.g.,
Hamilton-Jacobi-Bellman equations (as presented in 
\cite{SmearsSueli2014,SmearsSueli2016} for the nonsymmetric
formulation) is less obvious. This is due to the fact that the 
nonlinear operator in that problem does not have sufficient
smoothness properties that would allow an analysis of 
a direct least-squares procedure. 
Alternatively, the least-squares method could be applied to the 
linear problems from a semismooth Newton algorithm.
However, as semismoothness on the operator level does not hold
in general (cf.\ \cite[Remark~1]{SmearsSueli2014}), an analysis
of this method requires further investigation.

\paragraph{(b) Nonconvex domains}
The least-squares formulation may still be meaningful
on nonconvex domains, but the solution
will generally not coincide with that of \eqref{e:pdeL}.

\paragraph{(c) Nonconforming schemes}
Nonconforming finite elements for fourth-order problems
\cite{Ciarlet1978} have the advantage to be much simpler than
their conforming counterparts. Since discrete analogues of
\eqref{e:LaplHess} or \eqref{e:divrotestimate} may not be
satisfied without further stabilization terms, their application
would require further modifications.

\paragraph{(d) Lower-order terms}
Equations in nondivergence form including lower-order terms can
be equally well treated in the proposed framework,
see also \cite{SmearsSueli2014}.
The least-squares formulation
can be derived from the minimization of the functional
$
  \| A:D^2 u + b\cdot\nabla u +c u - f\|_{L^2(\Omega)}
$
for data $b$ and $c>0$.
For the mixed system \eqref{e:saddlepoint} this leads to a
coupling of equations \eqref{e:saddlepoint_a} and
\eqref{e:saddlepoint_b}--\eqref{e:saddlepoint_c}.
The Cordes condition with lower-order terms 
\cite{SmearsSueli2014} reads: there exists $\alpha>0$ such that
\begin{equation*}
 \left(|A|^2+|b|^2/(2\alpha)+(c/\alpha)^2\right)
 \big/
 (\trace A + c/\alpha)^2
\leq 1/(d+\varepsilon) .
\end{equation*}
More details on this Cordes condition are given in \cite{SmearsSueli2014}.

\paragraph{(e) Space dimensions higher than $d=3$}
The main arguments of this work are valid for any space dimension
$d\geq 2$. Also the mixed formulation can be formulated in any dimension,
provided it is posed in the space satisfying the constraint 
$\rot w =0$, which in higher dimensions is understood as
$Dw =(Dw)^*$. 
For the design of a numerical method,
it remains to identify the space $Q$ of multipliers.

\section*{Acknowledgments}
The author thanks Prof.~Ch.~Kreuzer for a helpful
discussion,
and the anonymous referees who helped
to significantly improve the presentation.
The author is funded
by Deutsche Forschungsgemeinschaft
(DFG) through CRC1173.

\providecommand{\bysame}{\leavevmode\hbox to3em{\hrulefill}\thinspace}
\providecommand{\MR}{\relax\ifhmode\unskip\space\fi MR }
\providecommand{\MRhref}[2]{%
  \href{http://www.ams.org/mathscinet-getitem?mr=#1}{#2}
}
\providecommand{\href}[2]{#2}

\end{document}